\documentclass[12pt]{amsart}
\usepackage{amsmath,verbatim,amsfonts,mathrsfs}
\usepackage{enumerate,color}
\usepackage{verbatim}% for comment
\usepackage{pdfsync}
\def\R{\mathbb{R}}

\def\bbbr{{\rm I\!R}}

\newcommand{\bh}{\mathbf{h}}

\newcommand{\bj}{\mathbf{j}}
\newcommand{\bjj}{\mathbf{J}}

\newcommand{\ba}{\mathbf{a}}
\newcommand{\bp}{\mathbf{p}}

\newtheorem{pro}{Proposition}[section]
\newtheorem{lem}[pro]{Lemma}
\newtheorem{thm}[pro]{Theorem}
\newtheorem{cor}[pro]{Corollary}

\newcounter{example}
\newcommand\spt{\mbox{\rm spt}}

\newcommand\pkd{\mbox{\rm dim}_{\rm P}\,} % packing dimension
\newcommand\hdd{\mbox{\rm dim}_{\rm H}\,} % Hausdorff dimension
\newcommand\ubd{\overline{\mbox{\rm dim}}_{\rm B}\,} % upper box dimension
\newcommand\lbd{\underline{\mbox{\rm dim}}_{\rm B}\,} % lower box dimension
\newcommand{\bu}{{\bf u}} % bold j
\newcommand{\bv}{{\bf v}} % bold j
\newcommand{\bw}{{\bf w}} % bold w
\newcommand{\w}{\omega} % random translate
\newcommand{\be}{\begin{equation}} % begin equation
\newcommand{\ee}{\end{equation}} % end equation

\newcommand{\red}[1]{{\color{red}#1}}

\begin{document}

\title{Generalized $q$-dimensions of measures on nonautonomous fractals}

\author{Yifei Gu}
\address{Department of Mathematics, East China Normal University, No. 500, Dongchuan Road, Shanghai 200241, P. R. China}

\email{52275500012@stu.ecnu.edu.cn}

\author{Jun Jie Miao}
\address{Department of Mathematics, East China Normal University, No. 500, Dongchuan Road, Shanghai 200241, P. R. China}

\email{jjmiao@math.ecnu.edu.cn}

\begin{abstract}
In the paper, we study the generalized $q$-dimensions of measures supported by nonautonomous attractors, which are the generalization of classic Moran sets and attractors of iterated function systems.

First, we estimate the generalized $q$-dimensions of measures supported on nonautonomous attractors, and we provide dimension formulas for generalized $q$-dimensions of measures supported on nonautonomous similar attractor under certain separation conditions. Next, we investigate the generalized $q$-dimensions of measures supported on nonautonomous affine sets and obtain the upper bounds. Finally, we study two variations of  nonautonomous affine sets and  obtain  their dimension formulas for  $q\geq 1 $.
\end{abstract}

\maketitle

\section{Introduction}
\subsection{Generalized $q$-dimensions}
The generalized $q$-dimensions (or $L^q$-dimensions)  of compactly supported Borel probability measures are  important concepts in fractal geometry and dynamical system which were introduced by Renyi in \cite{Renyi57,Renyi60} in the 1960s. Since the generalized $q$-dimensions  quantify global fluctuations of a given measure $\nu$, they provide valuable information about the multifractal properties of $\nu$ and also about the dimensions of its support.   In the case of a dynamical system, they are directly observable from the longtime behavior of obits, see \cite{Gr83, HP}. Moreover, they depend more regularly on the data, $\nu(B)$, rather than the mutltifractal spectrum of the measure $\nu$, and therefore, they easier to handle analytically and numerically, see~\cite{HJKP86,JKLPS85}.

Let $\nu$  be a positive finite Borel measure on $\mathbb{R}^d$. For $q \neq 1$,   the {\it lower} and {\it upper generalized $q$-dimensions} (also termed generalized R\'{e}nyi dimensions or $L^q$-{\it dimensions}) of $\nu$  are given by
\be
\underline{D}_q(\nu)=\liminf_{r\to 0}\frac{\log \sum_{\mathcal{M}_r}\nu(Q)^q}{(q-1)\log r}, \quad
\overline{D}_q(\nu)=\limsup_{r\to 0}\frac{\log \sum_{\mathcal{M}_r}\nu(Q)^q}{(q-1)\log r},
\label {gendim}
\ee
where $\mathcal{M}_r$ is  the family
of $r$-\textit{mesh cubes} in $\mathbb{R}^d$, that is cubes of the form
\newline
$[j_1 r, (j_1+1)r)\times \cdots\times [j_N r, (j_N+1)r)$ where $j_1,
\ldots, j_N \in \mathbb{Z}$.   For $q=1$, $\underline{D}^1(\nu)$ and $\overline{D}^1(\nu)$, also termed the
\textit{lower} and \textit{upper information dimensions}, are defined by
\be\label {infdim}
\underline{D}_1(\nu)=\liminf_{r\to 0}\frac{\sum_{\mathcal{M}_r}\nu(Q)\log\nu(Q)}{\log r}, \quad
\overline{D}_1(\nu)=\limsup_{r\to 0}\frac{\sum_{\mathcal{M}_r}\nu(Q)\log\nu(Q)}{\log r}.
\ee
If $\underline{D}_q(\nu)=\overline{D}_q(\nu)$,
we write $D_q(\nu)$ for the common value which we refer to as the
\textit{generalized $q$-dimension}. Note that the generalized $q$-dimension of a measure contains information about the measure
 and its support, and it directly follows from the definition that
$$
\ubd \spt(\nu)=\overline{D}_0(\nu), \quad \lbd \spt(\nu)=\underline{D}_0(\nu),
$$
where $\spt$ denotes the support of $\nu$. We refer readers to \cite{Bk_KJF2, P} for the background reading.

One concept related to generalized $q$-dimensions is $L^q$-spectrum. \textit{The lower and upper $L^q$-spectrum} of $\nu$ are defined by
\be
\underline{\tau}(\nu,q)=\liminf_{r\to 0}\frac{\log \sum_{\mathcal{M}_r}\nu(Q)^q}{-\log r}, \quad
\overline{\tau}(\nu,q)=\limsup_{r\to 0}\frac{\log \sum_{\mathcal{M}_r}\nu(Q)^q}{-\log r}.
\ee
If $\underline{\tau}(\nu,q)=\overline{\tau}(\nu,q)$, we write $\tau(\nu,q)$ for the common value which we call  \textit{$L^q$-spectrum}. It is clear that for $q\neq 1$,
$$
D_q(\nu)=\frac{\tau(\nu,q)}{1-q}
$$
if the $L^q$-spectrum $\tau(\nu,q)$ exists.  Furthermore, if $\overline{\tau}(\nu,q)$ is differentiable at $q=1$, then Ngai~\cite{Ngai97} showed that
$$
\hdd \nu = \pkd \nu = D_1(\nu) = - \overline{\tau}'(\nu,q).
$$
Moreover, Riedi proved that the Legendre transform is always equal to the $L^q$-spectrum $\tau(\nu,q)$, see ~\cite{Riedi95} for details.

Since generalized $q$-dimensions are closely related to various key concepts in fractal geometry, they have a rich literature concerning measures supported on fractal sets. In \cite{PSol1}, Peres and Solomyak proved that generalized $q$-dimensions exist for self-conformal measures, and Fraser \cite{Fraser16} extended to graph-directed self-conformal measures. Miao and Wu~\cite{MW15}  studied the generalized $q$-dimension of measures on the Heisenberg group. Qiu, Wang and Wang\cite{QWW24} estimate the upper and lower $L^q$-spectra of measures generated by a class of graph-directed planar non-conformal iterated function systems.  We refer readers to~\cite{BM07,BF13,BF21,CM,Gr83,HP,LN98,LN99,Ngai97,P} for various related studies on generalized $q$-dimensions.

\subsection{Iterated function systems}
Iterated function systems  are an important way to generate various fractals, such as self-conformal set and self-affine sets, and they also provide an important stage to explore the generalized $q$-dimensions.

An {\em iterated function system (IFS)} is a finite family of contractions
$S_1,\ldots,S_m: \bbbr^d \rightarrow \bbbr^d$ with $m\geq 2$. It is
well-known that there is a non-empty compact set $E\subseteq \bbbr^d$ such that
$$
E=\bigcup_{i=1}^m S_i(E),
$$
called the {\em attractor (or invariant set)} of the IFS.
If the contractions $S_1, S_2, \ldots, S_m$ are affine transformations, that is, with the form
$$S_i=T_i(x)+\w_i, \qquad x \in \mathbb{R}^d,$$
where $T_i$ is a linear transformation on $\mathbb{R}^d$ (representable by an $d \times d$ matrix) and $\w_i$
is a vector in $\mathbb{R}^d$, then the attractor $E$ is called a \textit{self-affine set}. If the contractions $S_1, S_2, \ldots, S_m$ are similarities, that is, with
$$|S_i(x) - S_i(y)| = c_i|x - y|, \qquad x, y \in \mathbb{R}^d,$$
where $0<c_i<1$, the attractor $E$ is called a \textit{self-similar set}. It is clear that self-similar sets are a particular case of self-affine sets.

Given a self-affine IFS $\{S_1, S_2, \ldots, S_m\}$ with the attractor $E$. Let $\bp=(p_1,\ldots,p_m)$ be a probability vector. The unique probability measure $\mu_\bp$ supported on $E$ satisfying
$$
\mu_\bp(\cdot)=\sum_{i=1}^m p_i \mu_\bp(S_i^{-1}(\cdot)),
$$
is called a \textit{self-affine measure}. If $S_1, S_2, \ldots, S_m$ are similarities, then $\mu_\bp$  is called a \textit{self-similar measure}. We refer readers to \cite{Bk_KJF2} for background reading.

Generalized $q$-dimensions, as well as other multifractal parameters, have already been investigated for various fractal measures supported on self-similar sets. Cawley and Mauldin~\cite{CM} studied multifractal analasis of  self-similar measures $\mu_\bp$ supported on self-similar set with the open set condition satisfied, and they proved that for $q>0$ and $q\neq 1$, the generalized $q$-dimension exists and
$$
D_q(\mu_\bp)=\frac{\tau(q)}{1-q},
$$
where $\tau(q)$ is given by
$$\sum_{i=1}^m p_i^q c_i^{\tau(q)}=1,$$
and for  $q= 1$,
$$
D_1(\nu) = - \tau(q)'=\frac{\sum_{i=1}^{m} p_i\log p_i}{\sum_{i=1}^{m} p_i\log c_i}=\hdd \mu_\bp.
$$

In the setting of self-affine measures, the generalized $q$-dimensions are notoriously difficult to compute. Currently, Most work focuses on the measures supported on nontypical self-affine sets, that is, the self-affine sets with grid structures such  as Bedford-McMullen carpets\cite{Bedfo84, McMul84}, Lalley-Gatzouras sets~\cite{LalGa92} and Baranski sets~\cite{Baran07}, see \cite{ DasSim17, Fraser16,GM22,KenPer96} for various examples. King~\cite{K} studied multifractal analysis of the self-affine measures supported on Bedford-McMullen carpets, where he showed that the generalized  $q$-exists and provided the dimension formula.  Olsen ~\cite{O} generalized King’s results to self-affine sponges which is a generalization of Bedford-McMullen carpets in higher dimensional space. Recently, Kolossv\'{a}ry~\cite{Kolo23} studied the $L^q$-spectrum of self-affine measures on sponges which satisfies a suitable separation condition. Feng and Wang~\cite{FW} calculated the generalized $q$-dimension of self-affine measures on the plane supported on attractors of IFS given by orientation-preserving diagonal matrices satisfying a suitable separation condition. Fraser~\cite{Fraser16} extended their result to include reflections and rotations by $90^\circ$.

Another class of self-affine sets is called  typical self-affine sets, where the translations in affine contractions are treated as parameters and almost sure results are often provided, see \cite{Falco88,Falco92}.  Falconer made important progress on generalized $q$-dimensions of measures supported on typical self-affine sets. In~\cite{Falco05}, he estimated the lower and upper bounds for  generalized $q$-dimensions of measures supported on typical self-affine sets, and he proved that  generalized $q$-dimension exists almost surely for $1<q\leq 2$. Later, Falconer~\cite{Falco10} extended the study to almost self-affine sets, see \cite{Falco10, JPS07} and he proved that generalized $q$-dimension exists for $q>1$. We refer the reader to \cite{BHR19,FM07,Fraser16, FK18,FLLMY24,LN98,LN99, Morr17,Ngai97} for various related studies and references therein.

\subsection{Nonautonomous iterated function systems}
Nonautonomous iterated function systems may be regarded as a generalization of iterated function systems. We give a brief review of Nonautonomous iterated function systems, and we refer readers to \cite{GM,GM22} for details. %In this paper, we study the  generalized $q$-dimensions of measures supported on the attractors generated by  nonautonomous iterated function systems.

Let $\{n_{k}\}_{k=1}^\infty$ be a sequence of integers greater than or equal to $2$. For
each $k=1,2,\cdots$, we write
\begin{equation}\label{defsigmak}
\Sigma^{k}=\{u_{1}u_{2}\cdots u_{k}:1\leq u_{j}\leq
n_{j},\ j\leq k\}
\end{equation}
for the set of words of length $k$, with $
\Sigma^{0}=\{\emptyset \}$ containing only the empty word $\emptyset$, and write
\begin{equation}\label{defsigma*}
\Sigma^{*}=\bigcup _{k=0}^{\infty }\Sigma^{k}
\end{equation}
for the set of all finite words. Let $\Sigma^\infty= \{(u_1u_2 \ldots u_k\ldots) : 1 \leq  u_k \leq n_k\}$ be the corresponding set of infinite words.

Suppose that $J\subset \R^d$ is a compact set with non-empty interior.
Let $\{\Xi_k\}_{k=1}^\infty$ be a sequence of collections of contractive mappings, that is
\begin{equation}\label{Xi}
\Xi_k=\{S_{k,1},S_{k,2},\ldots,S_{k,n_k}\},
\end{equation}
where each $S_{k,j}$ satisfies that
$|S_{k,j}(x)-S_{k,j}(y)| \leq c_{k,j} |x-y| $ for some $0<c_{k,j}<1$.
We say the collection $\mathcal{J}=\{J_{\mathbf{u}}:\mathbf{u}\in \Sigma^*\}$ of closed subsets of $J$ fulfils the \textit{nonautonomous structure with respect to $\{\Xi_k\}_{k=1}^\infty$} if it satisfies the following conditions:
\begin{itemize}
\item[(1).] For all integers $k>0$ and all $\mathbf{u}\in \Sigma^{k-1}$, the elements $J_{\mathbf{u}1}, J_{\mathbf{u}2},\cdots, J_{\mathbf{u}n_{k}}$ of $\mathcal{J}$ are the subsets of $J_{\mathbf{u}}$. We write $J_{\emptyset }=J$ for the empty word $\emptyset $.
\item[(2).] For each $\mathbf{u}=u_1\ldots u_k \in \Sigma^*$,  there exists an  transformation $\Psi_\mathbf{u}: \mathbb{R}^{d}\rightarrow \mathbb{R}^{d}$  such that
    $$
    J_{\mathbf{u}}=\Psi_{\mathbf{u}}(J)=\Psi_{u_1}\circ \ldots \circ \Psi_{u_j} \ldots \circ \Psi_{u_k} (J),
    $$
    where $\Psi_{u_j}(x)=S_{j,u_j}x+ \w_{u_1\ldots u_j}$, for some $\w_{u_1\ldots u_j} \in \R^d$, and $S_{j,u_j}\in \Xi_j$, $j=1,2,\ldots k$.
\item[(3).]The maximum of the diameters of $J_\bu$ tends to 0 as $|\bu|$ tends to $\infty$, that is,
    $$
    \lim_{k\to \infty} \max_{\bu\in \Sigma^k} |J_\bu|=0.
    $$

\end{itemize}
We call $\{\Xi_k\}_{k=1}^\infty$ \textit{ a nonautonomous iterated function systems}(NIFS) determined by $\mathcal{J}$, and the non-empty compact set
\begin{equation}\label{attractor}
E=E(\mathcal{J})=\bigcap\nolimits_{k=1}^{\infty }\bigcup\nolimits_{\mathbf{u}\in \Sigma^{k}}J_{\mathbf{u}}
\end{equation}
is called a \textit{nonautonomous attractor} determined by $\mathcal{J}$.  For all $\mathbf{u}\in \Sigma^{k}$, the elements $J_{\mathbf{u}}$ are called \textit{\ $k$th-level basic sets} of $E$.

 If the nonautonomous attractor $E$ satisfies  that for all integers $k>0$ and $\mathbf{u}\in \Sigma^{k-1}$,
$$\mathrm{int} (J_{\bu i})\cap \mathrm{int} (J_{\bu i'}) = \emptyset  \quad \textit{ for } i\neq i'\in\{1,2,\ldots, n_k\},
$$
we say $E$ satisfies \textit{open set condition} (OSC).  If for all $k\in \mathbb{N}$ and $\mathbf{u}\in \Sigma^{k-1}$,
$$
J_{\bu i}\cap J_{\bu i'} = \emptyset  \quad \textit{ for } i\neq i'\in\{1,2,\ldots, n_k\},
$$
then we say the nonautonomous attractor  $E$ satisfies \textit{strong separation condition}(SSC).

If the contractions in $\Xi_k$  are linear mappings, that is
\begin{equation}\label{Xi}
\Xi_k=\{T_{k,1},T_{k,2},\ldots,T_{k,n_k}\},
\end{equation}
where
$T_{k,j}$ are $d\times d$ matrices with $\|T_{k,j}\|<1$ for $j=1,2,\ldots n_k$, then we call $\{\Xi_k\}_{k=1}^\infty$  the \textit{nonautonomous affine iterated function system}(NAIFS) and call $E$ a \textit{nonautonomous affine set}(NAS). If for each $k$,  $T_{k,1},\ldots T_{k,n_k}$ are similarities, that is, with
$$
|T_{k,i}(x)-T_{k,i}(y)|=c_{k,i}|x-y|
$$
where $c_{k,i}<1$, then we call  $\{\Xi_k\}_{k=1}^\infty$  the \textit{nonautonomous similar iterated function system }(NSIFS) and call $E$ a \textit{nonautonomous similar set}(NSS).  Note that nonautonomous attractors may also be regarded as a generalization  of Moran sets, see~\cite{GM22,Moran,Wen00} for details,  and Moran sets are a special class of nonautonomous similar sets satisfying open set condition.

For $\bu=u_1\cdots u_k\in\Sigma^k$, we write $\bu^-=u_1\ldots u_{k-1}$ and  write $|\bu|=k$ for the length of $\bu$. For each $\bu = u_1 u_2\cdots u_k\in
\Sigma^{*}$, and $\bv = v_1 v_2\cdots \in \Sigma^{\infty}$,   we say $\mathbf{u}$ is a \textit{curtailment or prefix } of $\bv$, denote by  $\mathbf{u} \preceq \bv$, if $\mathbf{u} = v_1\cdots v_k =\bv|k$. We call the set $\mathcal{C}_\bu =
\{\bv\in\Sigma^{\infty} : \bu \preceq\bv\}$ the \textit{cylinder} of $\mathbf{u}$. If $\mathbf{u}=\emptyset$, its cylinder is $\mathcal{C}_\bu=\Sigma^{\infty}$.  We term a subset $A$ of $\Sigma^*$ a	\textit{cut set} if $\Sigma^\infty\subset\bigcup_{\mathbf{u}\in A}\mathcal{C}_\mathbf{u}$, where $\mathcal{C}_\mathbf{u}\bigcap\mathcal{C}_{\mathbf{v}}=\emptyset$ for  all $\mathbf{u}\neq\mathbf{v}\in A$. It is equivalent to that, for every $\mathbf{w}\in \Sigma^\infty$, there is a unique  $\mathbf{u}\in A$ with $|\mathbf{u}|<\infty$ such that $\mathbf{u}\preceq \mathbf{w}$.  If
$\bu,\bv\in \Sigma^\infty$, then $\bu\wedge\bv \in \bjj$ denotes the
maximal common initial subword of both $\bu$ and $\bv$.

For $\bv=v_1\ldots v_k\in \Sigma^k$, we denote compositions of mappings by $S_\mathbf{v}= S_{1,v_1}\ldots S_{k,v_k}$. Let $\Pi: \Sigma^\infty \rightarrow \mathbb{R}^d$ be the projection given by
\begin{eqnarray}
\Pi(\bu) & =& \bigcap_{k=0}^\infty (S_{1,u_1} + \w_{u_1})  (S_{2,u_2} + \w_{u_1u_2}) \cdots (S_{k,u_k} + \w_{\bu|k})(J)  \nonumber\\
& =& \lim_{k \to \infty} (S_{1,u_1} + \w_{u_1})  (S_{2,u_2} + \w_{u_1u_2}) \cdots (S_{k,u_k} + \w_{\bu|k})(x)\label{points2}\\
& =&  \w_{u_1} +S_{1,u_1}\w_{u_1u_2}+ \cdots +  S_{\bu|k}\w_{\bu|k+1}  + \cdots.  \nonumber
\end{eqnarray}
It is clear that the attractor $E$ is the image of $\Pi$, i.e. $E=\Pi(\Sigma^\infty)$. To emphasize the dependence on translations, we sometimes write $\Pi^\w(\bu)$ and $E^\w$ instead of $\Pi(\bu)$ and $E$.

Let $\mu$  be  a positive finite Borel measure on $\Sigma^\infty$. We define  $\mu^\w$, the projection of  measure $\mu$ onto $\R^d$, by
\begin{equation}
 \mu^\w(A)=\mu\{\bu: \Pi^\w(\bu)\in A\},                    \label{def_ua}
\end{equation}
for $A\subseteq\mathbb{R}^d$, or equivalently by
\be \int
f(x)d\mu^\w(x)=\int f(\Pi^\w(\bu))d\mu(\bu),  \label{def_ua2} \ee
for every continuous  $f:\R^d\to \mathbb{R}$. Then $\mu^\w$ is a Borel measure supported on $E^{\w}$.

In particular, given a sequence of probability vectors $\{\bp_k=(p_{k,1},\ldots,p_{k,n_k})\}_{k=1}^\infty$, that is, $\sum_{i=1}^{n_k} p_{k,i}=1$ for all $k>0$, we may define a Bernoulli measure $\mu$ on $\Sigma^\infty$ by setting
\begin{equation}\label{mP}
\mu(\mathcal{C}_\bu)=p_\bu\equiv p_{k,u_1}p_{k,u_2}\ldots p_{k,u_k} \quad (\bu=u_1\ldots u_k)
\end{equation}
for each cylinder $\mathcal{C}_\bu$ and
extending to general subsets of $\Sigma^\infty$ in the usual way.  With $E^\w$ a nonautonomous affine(similar) set, the projected measure $\mu^\w$ on $E^{\w}$  given by (\eqref{def_ua})  is termed a {\it nonautonomous affine(similar) measure}.

It is nature to explore the properties of generalized $q$-dimensions of measures supported on nonautonomous attractors. However, such sets, different to the classic attractors generated by IFS, do not have any dynamical properties anymore, and the tools of ergodic theory cannot be invoked. Therefore generalized $q$-dimensions  rarely exist, and it is more difficult to determine their relevant  properties. In this paper, we investigate the lower and upper generalized $q$-dimensions of measures supported on nonautonomous attractors, and  we state the main conclusions
in Section ~\ref{sec 2}. In Section ~\ref{sec 3}, we estimated the upper and lower bounds for the generalized $q$-dimensions of measures supported on nonautonomous attractors, and dimension formulas are provided for  nonautonomous  similar measures.  In Section ~\ref{sec 4}, we study the upper bounds for generalized $q$-dimensions of measures supported on nonautonomous affine sets. In Section ~\ref{sec 5} and Section ~\ref{sec_6}, we study the lower bounds for generalized $q$-dimensions of measures supported on two variations of nonautonomous affine attractors.

\section{Generalized $q$-dimensions of measures on Non-autonomous fractals}\label{sec 2}

\subsection{Generalized $q$-dimensions of measures on nonautonomous fractals}

Let $\Sigma^*$ be given by \eqref{defsigma*}. Given a sequence of positive real vectors $\{(c_{k,1},c_{k,2},\cdots, c_{k,n_{k}})\}_{k=1}^\infty$. In this paper, we always assume that
$$
c_*=\inf_{k,j} \{c_{k,j}\} > 0, \qquad\qquad c^*=\sup_{k,j} \{c_{k,j}\} < 1.
$$

For each $s>0$ and $0<r<1$, we write
$$
\Sigma^*(s, r)=\{\mathbf{u}=u_1\ldots u_k \in \Sigma^*: c_\mathbf{u}  \leq r < c_{\bu^-} \}.
$$
It is clear that $\Sigma^*(s,r)$ is a  cut-set, and for each $\mathbf{u}\in \Sigma^*(s,r)$,  we have that
$$
c_{*} r < c_\mathbf{u} \leq r.
$$
Let $\mu$ be a  positive finite Borel measure on $\Sigma^\infty$. We define that
\begin{equation}\label{def_dq-}
    d_q^- = \left\{
    \begin{aligned}
    & \sup \Big\{s: \limsup_{r\to 0}\sum_{\bu\in\Sigma^*(s, r)}c_\bu^{s(1-q)}\mu(\mathcal{C}_\bu)^q<\infty\Big\}, \qquad\quad\text{ for } q>1;\\
    & \inf \Big\{s: \liminf_{r\to 0}\sum_{\bu\in\Sigma^*(s, r)}c_\bu^{s(1-q)}\mu(\mathcal{C}_\bu)^q<\infty\Big\}, \qquad\quad  \  \ \text{ for }  0<q<1;\\
    & \sup \Big\{s: \limsup_{r\to 0}\sum_{\bu\in\Sigma^*(s, r)}\mu(\mathcal{C}_\bu)\log\Big(c_\bu^{-s}\mu(\mathcal{C}_\bu)\Big)<\infty\Big\}, \text{ for } q=1.
    \end{aligned}
    \right.
\end{equation}
and
\begin{equation}\label{def_dq+}
    d_q^+ = \left\{
    \begin{aligned}
    & \sup \Big\{s: \liminf_{r\to 0}\sum_{\bu\in\Sigma^*(s, r)}c_\bu^{s(1-q)}\mu(\mathcal{C}_\bu)^q<\infty\Big\}, \qquad\quad\text{ for } q>1;\\
    & \inf \Big\{s: \limsup_{r\to 0}\sum_{\bu\in\Sigma^*(s, r)}c_\bu^{s(1-q)}\mu(\mathcal{C}_\bu)^q<\infty\Big\},  \qquad\quad \text{ for }  0<q<1;\\
    & \sup \Big\{s: \liminf_{r\to 0}\sum_{\bu\in\Sigma^*(s, r)}\mu(\mathcal{C}_\bu)\log\Big(c_\bu^{-s}\mu(\mathcal{C}_\bu)\Big)<\infty\Big\}, \text{ for } q=1.
    \end{aligned}
    \right.
\end{equation}

Generally, it is difficult to find the generalized $q$-dimensions of measures supported on nonautonomous attractors, but we are still able to provide some rough estimates if contraction ratios are known. First, we show that $d_q^-$ and $d_q^+$ are the natural upper bounds for the lower and upper  generalized $q$-dimensions.
\begin{thm}\label{thm_DqU}
Let $E$ be the nonautonomous attractor given by~\eqref{attractor}. Let $\mu$ be a positive finite Borel measure on $\Sigma^\infty$, and let $\mu^\w$ be the image measure of $\mu$ given by~\eqref{def_ua}. For each integer $ k>0$, we assume that
$$
|S_{k,i}(x)-S_{k,i}(y)| \leq c_{k,i}|x-y|, \qquad \textit{ for }x,y \in J,
$$
where $c_{k,i}<1$, for all $i=1,2,\ldots, n_k$. Then  for all $q>0$,
$$
\underline{D}_q(\mu^\w)\leq\min\{d_q^-,d\}, \quad \textit{and } \quad
\overline{D}_q(\mu^\w)\leq \min\{d_q^+,d\},
$$
where $d_q^-$ and $d_q^+$ are given by \eqref{def_dq-} and \eqref{def_dq+} with respect to $(c_{k,1},c_{k,2},\cdots, c_{k,n_{k}})$.
\end{thm}

For the lower bounds, it is more awkward to compute, and they often heavily depend on the geometric structures of the supported sets. We say the nonautonomous set $E$ satisfies {\it gap separation condition(GSC)} if there exists a constant $C$ such that for all $\bu\in\Sigma^*$, we have that
$$
%d(J_{\bu i},J_{\bu j}) \geq C |J_\bu|.
\inf\{|x-y|: x\in J_{\bu i}, y\in J_{\bu j },  i\neq j\}\geq C|J_\bu|.
$$
Note that GSC implies the strong separation condition.  Under this extra assumption, we are able to show that $d_q^-$ and $d_q^+$ are also the lower bounds for the lower and upper  generalized $q$-dimensions.
\begin{thm}\label{thm_DqL}
Let $E$ be the nonautonomous attractor given by~\eqref{attractor} with gap separation condition satisfied. Let $\mu$ be a positive finite Borel measure on $\Sigma^\infty$, and let $\mu^\w$ be the image measure of $\mu$ given by~\eqref{def_ua}. For each integer $ k>0 $, we assume that
$$
|S_{k,i}(x)-S_{k,i}(y)| \geq c_{k,i}|x-y|, \qquad \textit{ for }x,y \in J,
$$
where $c_{k,i}<1$,  for all $i=1,2,\ldots, n_k$. Then  for all $q>0$,
$$
\underline{D}_q(\mu^\w) \geq  d_q^-, \quad \textit{ and } \quad  \overline{D}_q(\mu^\w) \geq d_q^+,
$$
where $d_q^-$ and $d_q^+$ are given by \eqref{def_dq-} and \eqref{def_dq+} with respect to $(c_{k,1},c_{k,2},\cdots, c_{k,n_{k}})$.
\end{thm}

Immediately, we have the following conclusion.
\begin{cor}\label{proc}
Let $E$ be the nonautonomous attractor given by~\eqref{attractor} with gap separation condition satisfied. Let $\mu$ be a positive finite Borel measure on $\Sigma^\infty$, and let $\mu^\w$ be the image measure of $\mu$ given by~\eqref{def_ua}. For each integer $ k>0 $, we assume that
$$
 c_{k,i}'|x-y|\leq |S_{k,i}(x)-S_{k,i}(y)| \leq  c_{k,i}|x-y|, \qquad \textit{ for }x,y \in J,
$$
where $c_{k,i}<1$,  for all $i=1,2,\ldots, n_k$.  Let $d_q^-$ be given by \eqref{def_dq-} with respect to $(c_{k,1}',\ldots,c_{k,n_k}')$ and let $d_q^+$ be given by  \eqref{def_dq+} with respect to $(c_{k,1},\ldots,c_{k,n_k})$.  Suppose that  $d_q^-=d_q^+=d_q$  for all $q>0$. Then for all $q>0$, the generalized $q$-dimension $D_q(\mu^\w)$ exists and
$$
D_q(\mu^\w)=\min\{d_q,d\}.
$$
\end{cor}

Recall that $E$ is a nonautonomous similar set, that is, with
$$
 |S_{k,i}(x)-S_{k,i}(y)| =  c_{k,i}|x-y|, \qquad  \textit{ for }x,y \in J,
$$
where $c_{k,i}<1$,  for all $i=1,2,\ldots, n_k$ and for all $k>0 $.  Due to the excellent geometric properties of  nonautonomous similar sets,  we are able to find the  generalized $q$-dimensions formulas under SSC instead of GSC.

\begin{thm}\label{thm_MS}
Let $E$ be a nonautonomous similar set given by ~\eqref{attractor} with strong separation condition satisfied. Let $\mu$ be  a positive finite Borel measure on $\Sigma^\infty$, and let $\mu^\w$ be the image measure of $\mu$ given by~\eqref{def_ua}. Then for all $q>0$,
we have
$$
\underline{D}_q(\mu^\w)=\min\{d_q^-,d\}, \quad \textit{ and }\quad
\overline{D}_q(\mu^\w)=\min\{d_q^+,d\}.
$$
\end{thm}

To find the generalized $q$-dimensions of measures supported on nonautonomous similar sets, the key is to show the identical of  $d_q^-$ and $d_q^+$ given by  \eqref{def_dq-} and \eqref{def_dq+} with respect to contraction ratios $(c_{k,1},c_{k,2},\cdots, c_{k,n_{k}})$, and the following conclusion is a direct consequence of Theorem \ref{thm_MS}.
\begin{cor}\label{corNAss}
Let $E$ be a nonautonomous similar set  given by ~\eqref{attractor} with strong separation condition satisfied. Let $\mu$ be a positive finite Borel measure on $\Sigma^\infty$, and let $\mu^\w$ be the image measure of $\mu$ given by~\eqref{def_ua}.   Suppose that  $d_q^-=d_q^+=d_q$  for all $q>0$.
Then for all $q>0$, the generalized $q$-dimension exists and
$$
D_q(\mu^\w)=\min\{d_q,d\}.
$$
\end{cor}

Moreover, if $\mu$ is a Bernoulli measure on $\Sigma^\infty$ define by \eqref{mP} with a sequence of probability vectors $\{\bp_k\}_{k=1}^\infty$,  then  the formulas of the critical values $d_q^-$ and $d_q^+$ are simplified into the following forms.

\begin{pro}\label{leq_dq}
Let $E$ be a nonautonomous similar set, and let $\mu$ be a Bernoulli measure on $\Sigma^\infty$ define by \eqref{mP} with a sequence of probability vectors $\{\bp_k\}_{k=1}^\infty$. Then for $q>1$,
\begin{eqnarray}
d_q^- &=& \sup \Big\{s: \limsup_{k\to \infty} \prod_{i=1}^k \sum_{j=1}^{n_i}c_{i,j}^{s(1-q)}p_{i,j}^q<\infty\Big\}, \label{dq-eq}\\
d_q^+ &\leq& \sup \Big\{s: \liminf_{k\to \infty}\ \prod_{i=1}^k \sum_{j=1}^{n_i}c_{i,j}^{s(1-q)}p_{i,j}^q<\infty\Big\} \label{dq+leq1};
\end{eqnarray}
for $0 < q < 1$,
\begin{eqnarray}
d_q^- &\geq& \inf \Big\{s: \liminf_{k\to \infty} \prod_{i=1}^k \sum_{j=1}^{n_i}c_{i,j}^{s(1-q)}p_{i,j}^q<\infty\Big\}, \label{dq-geq1}\\
d_q^+ &=& \inf \Big\{s: \limsup_{k\to \infty} \prod_{i=1}^k \sum_{j=1}^{n_i}c_{i,j}^{s(1-q)}p_{i,j}^q<\infty\Big\}. \label{dq+eq}
\end{eqnarray}
and for $q=1$,
\begin{eqnarray}
d_q^- &\geq& \sup \Big\{s: \limsup_{k\to \infty} \sum_{i=1}^k \sum_{j=1}^{n_i} \big(p_{i,j}\log p_{i,j} - s\ p_{i,j}\log c_{i,j}\big)<\infty\big\}; \label{dq-geq2}\\
d_q^+ &\leq& \sup \Big\{s: \liminf_{k\to \infty}\sum_{i=1}^k \sum_{j=1}^{n_i} \big(p_{i,j}\log p_{i,j} - s\ p_{i,j}\log c_{i,j}\big)<\infty\Big\}. \label{dq+leq2}
\end{eqnarray}
\end{pro}

Since the nonautonommous similar measure $\mu^\w$ is the image measure of the Bernoulli measure $\mu$,
the above proposition gives a simpler formula for the generalizaed $q$-dimensions of nonautonommous similar measures.

Given a nonautonomous similar set $E$, even if all  $\Xi_{k}$ are identical, that is,
$$
\Xi_{k}=\Xi=(c_{1},c_{2},\cdots, c_{n_0}),
$$
the attractor $E$ is not necessarily a self-similar set since the translations at each level may still vary differently for the same contract ratio, see \cite{GM22,Wen00} for details. For such sets $E$, the critical values  $d_q^+$ and $d_q^-$ are equal, and the generalized $q$-dimensions  of  nonautonomous similar measures supported on $E$ may be  easily computed in the following conclusion.

 \begin{cor}\label{cor_MS}
Let $E$ be a nonautonomous similar set  given by ~\eqref{attractor} with  $\Xi_{k}=\Xi=(c_{1},c_{2},\cdots, c_{n_0})$ for all $k>0$. Let $\mu^\w$ be the nonautonomous similar measure  with $\bp_k=\bp=(p_{1},p_{2},\cdots, p_{n_0})$ for all $k>0$. Then $d_q^-=d_q^+=d_q$ for $q>0$,
where $d_q$ is the unique solution to
$$
\sum_{i=1}^{n_0} c_i^{d_q(1-q)} p_i^q=1 \quad   \textit{  for $q\neq 1$  \qquad  and  \qquad } %for $q= 1,$  }
d_1=\frac{\sum_{i=1}^{n_0} p_i \log p_i }{\sum_{i=1}^{n_0} p_i \log c_i }.
$$
Moreover, suppose that $E$ satisfies SSC. Then for $q>0$,  the generalized $q$-dimension $D_q(\mu^\w)$ exists and
$$
D_q(\mu^\w)=\min\{d_q,d\}.
$$
\iffalse
$$
\sum_{i=1}^{n_0} p_i \log (c_i^{-s} p_i) = 0,
$$
for $q=1$.

\fi
\end{cor}

\subsection{Generalized  $q$-dimensions of measures on nonautonomous affine sets}
Since affine mappings have different contraction ratios in different directions, the geometric properties of nonautonomous affine sets are more complex than nonautonomous similar sets. Although we have obtained some bounds for generalized $q$-dimensions of measures supported on nonautonomous attractors, these bounds may significantly be improved for nonautonomous affine sets by applying singular value functions.

Let $T\in\mathcal{L}(\R^d,\R^d)$ be a contracting and non-singular linear mapping. The {\em singular values} $\alpha_1(T),\alpha_2(T),\ldots$, $\alpha_d(T)$ of $T$ are the lengths of the (mutually perpendicular) principle semi-axes of $T(B)$, where $B$ is the unit ball in $\R^d$. Equivalently they are the positive square roots of the eigenvalues of $T^\ast T$, where
$T^\ast$ is the transpose of $T$. Conventionally, we write that
$$
1>\alpha_1(T)\geq \ldots\geq \alpha_d(T)>0.
$$
For $0\leq s \leq d$, the \textit{singular value function} of $T$ is defined by
\begin{equation}\label{def svf}
\psi^s(T)=\alpha_1(T)\alpha_2(T)\ldots \alpha_{m-1}(T) \alpha_m^{s-m+1}(T),
\end{equation}
where $m$ is the integer such that $m-1<s\leq m$. If $s > d$, for technical convenience,  we write
$$
\psi^s(T)= (\alpha_1(T)\alpha_2(T)\ldots \alpha_d(T))^{s/d}=(\det T)^{s/d}.
$$
The singular value function $\psi^s(T)$ is continuous and strictly decreasing in $s$, and it  is submultiplicative in $T$, that is, for all $s\geq 0$,
\begin{equation}\label{subphi}
\psi^s(TU)\leq \psi^s(T)\psi^s(U),
\end{equation}
for all $T,U\in \mathcal{L}(\R^d,\R^d)$, see ~\cite{Falco88} for details.

Let $\{\Xi_k\}_{k=1}^\infty$ be given by \eqref{Xi}. We write
\begin{eqnarray}\label{alpha}
\alpha_+&=&\sup\{\alpha_1(T_{k,i}) : 1\leq i\leq n_k  \textit{,  } k \in \mathbb{N}\}, \\
\alpha_-&=&\inf\{\alpha_d(T_{k,i}) : 1\leq i\leq n_k  \textit{,  } k \in \mathbb{N}\}. \nonumber
\end{eqnarray}
Immediately, for each $\bu\in \Sigma^k$, the singular value function $\psi^s$ of $T_\bu$ is bounded by
\begin{equation}\label{alhpha+-}
\alpha_-^{sk}\leq \psi^s(T_\bu) \leq \alpha_+^{sk}.
\end{equation}

     From now on, we always assume that the matrices in $\Xi_k$ for all $ k \in \mathbb{N}$ are nonsingular, and
\begin{equation} \label{aspalpha}
0< \alpha_- \leq   \alpha_+<1.
\end{equation}

Given reals $s>0$ and $0<r <1$. Let $m$ be the integer such that $m-1<s\leq m$, and we define
\begin{equation} \label{srcutsetaff}
\Sigma^*(s, r)=\{\mathbf{u}=u_1\ldots u_k \in \Sigma^*: \alpha_m(T_\mathbf{u})  \leq r < \alpha_m(T_{\bu^-}) \}.
\end{equation}
The set $\Sigma^*(s, r)$ is a cut-set of $\Sigma^\infty$. For each $\mathbf{u}\in \Sigma^*(s, r)$, by \eqref{subphi} and \eqref{alpha},  we have that
$$
\alpha_- r < \alpha_m(T_\mathbf{u}) \leq r.
$$

Similarly, we write
\begin{equation}\label{def_adq-}
    d_q^- = \left\{
    \begin{aligned}
    & \sup \Big\{s: \limsup_{r\to 0}\sum_{\bu\in\Sigma^*(s, r)}\psi^s(T_\bu)^{1-q}\mu(\mathcal{C}_\bu)^q<\infty\Big\},\qquad\qquad  \text{ for } q>1;\\
    & \inf \Big\{s: \liminf_{r\to 0}\sum_{\bu\in\Sigma^*(s, r)}\psi^s(T_\bu)^{1-q}\mu(\mathcal{C}_\bu)^q<\infty\Big\},  \qquad\qquad\ \  \text{ for }  0<q<1;\\
    & \sup \Big\{s: \limsup_{r\to 0}\sum_{\bu\in\Sigma^*(s, r)}\mu(\mathcal{C}_\bu)\log\Big(\psi^s(T_\bu)^{-1} \mu(\mathcal{C}_\bu)\Big)<\infty\Big\}, \! \text{ for } q=1.
    \end{aligned}
    \right.
\end{equation}
and
\begin{equation}\label{def_adq+}
    d_q^+ = \left\{
    \begin{aligned}
    & \sup \Big\{s: \liminf_{r\to 0}\sum_{\bu\in\Sigma^*(s, r)}\psi^s(T_\bu)^{1-q}\mu(\mathcal{C}_\bu)^q<\infty\Big\}, \qquad\qquad\ \text{ for } q>1;\\
    & \inf \Big\{s: \limsup_{r\to 0}\sum_{\bu\in\Sigma^*(s, r)}\psi^s(T_\bu)^{1-q}\mu(\mathcal{C}_\bu)^q<\infty\Big\},  \qquad\qquad\  \text{ for }  0<q<1;\\
    & \sup \Big\{s: \liminf_{r\to 0}\sum_{\bu\in\Sigma^*(s, r)} \mu(\mathcal{C}_\bu) \log\Big(\psi^s(T_\bu)^{-1}\mu(\mathcal{C}_\bu)\Big)<\infty\Big\}, \text{ for } q=1.
    \end{aligned}
    \right.
\end{equation}

The following theorem shows that $d_q^+$ and $d_q^-$ are respectively upper bounds for the upper and lower generalized $q$-dimension of  Borel measures supported on nonautonomous affine sets.
\begin{thm}\label{Dq_ub}
Let $E^\w$ be the nonautonomous affine set given by ~\eqref{attractor}. Let $\mu$  be a positive finite Borel measure on  $\Sigma^\infty$, and let $\mu^\w$ be the image measure of $\mu$ given by~\eqref{def_ua}. Then
$$
\overline{D}_q(\mu^\w) \leq \min\{d_q^+,d\}, \quad \textit{ and }\quad
\underline{D}_q(\mu^\w) \leq \min\{d_q^-,d\}
$$
 for all $\w$.
\end{thm}
We do not have general results for the lower bounds for  the generalized $q$-dimensions in nonautonomous affine settings except the rough one by applying Theorem~\ref{thm_DqL} with the fact $1>\alpha_1(T)\geq \ldots\geq \alpha_d(T)>0$. Therefore we have to consider some variations of nonautonomous affine sets to obtain better lower bounds for generalized $q$-dimensions.

\subsection{Generalized $q$-dimensions of measures on two variations of nonautonomous affine attractors}

It is a challenging problem to find appropriate lower bounds for generalized $q$-dimensions of measures supported on nonautonomous affine attractors.  Inspired by~\cite{Falco10,JPS07,GM}, we consider nonautonomous affine sets in probabilistic language.

Let $\mathcal{D}$ be a bounded region in $\R^d$. For each $\bu\in \Sigma^* $, let $\omega_\bu\in  \mathcal{D} $ be a random vector distributed
according to some Borel probability measure $P_\bu$ that is absolutely continuous with respect to $d$-dimensional Lebesgue measure. We assume that the $\omega_\bu$ are independent identically
distributed random vectors. Let $\mathbf{P} $ denote the product probability measure $\mathbf{P} = \prod_{\bu\in \Sigma^*} P_\bu$ on
the family $\omega =\{\omega_\bu : \bu \in  \Sigma^* \}.$
In this context, for each $\mathbf{u}=u_1\ldots u_k \in \Sigma^*$,  we assume that the translation of $\Psi_{u_j}$ is an element of $\w$, that is,
\begin{equation}
\Psi_{u_j}(x)=T_{j,u_j}x+ \w_{u_1\ldots u_j}, \qquad \w_{u_1\ldots u_j}\in \omega =\{\omega_\bu : \bu \in  \Sigma^* \},
\end{equation}
 for $j=1,2,\ldots, k$. We also assume that the collection of  $J_{\mathbf{u}}=\Psi_{\mathbf{u}}(J)=\Psi_{u_1}\circ  \ldots \circ \Psi_{u_k} (J)$ fulfils the nonautonomous structure, and we call $E^\omega$ the {\it nonautonomous affine set with random translations}.

Next we shows that $d_q^-$ defined by \eqref{def_adq-} is an almost sure lower bound for the lower generalized $q$-dimension of measures on nonautonomous affine sets with random translations for $q>1$.

\begin{thm}\label{thm_Dq1}
Let $E^\w$ be the nonautonomous affine set with random translations. Let $\mu$ be  a positive finite Borel measure on $\Sigma^\infty$, and let $\mu^\w$ be the image measure of $\mu$ given by~\eqref{def_ua}.
Then for $q>1$, we have that for almost all $\w$,
$$
\underline{D}_q(\mu^\w)=\min\{d_q^-,d\},
$$
where
$$
d_q^-=\sup \Big\{s: \sum_{k=1}^\infty \sum_{\bu\in\Sigma^k} \psi^s(T_\bu)^{1-q}\mu(\mathcal{C}_\bu)^q<\infty\Big\}.
$$
\end{thm}

Combining this with Theorem \ref{Dq_ub}, we obtain  a sufficient condition for  the existence of generalized $q$-dimension.
\begin{pro}\label{Dq_eq}
Let $E^\w$ be the nonautonomous affine set with random translations. Let $\mu$ be  a positive finite Borel measure on $\Sigma^\infty$, and let $\mu^\w$ be the image measure of $\mu$ given by~\eqref{def_ua}. Suppose that $d_q^-=d_q^+=d_q$. Then for $q>1$,  the generalized $q$-dimension $D_q(\mu^\w)$ exists and
$$
D_q(\mu^\w)=\min\{d_q,d\},
$$
for almost all $\w$.
\end{pro}

If all  $\Xi_{k}$ are identical, that is,
$$
\Xi_{k}=\Xi=\{T_{1},T_{2},\ldots,T_{n_0}\},
$$
for all integers $k>0 $,
we write $d_q$ for  the unique solution to
\begin{equation}\label{def_sadg}
\lim_{k\to\infty} \Big(\sum_{\bu\in\Sigma^k} \psi^{d_q}(T_\bu)^{1-q}  \  p_\bu^q \Big)^{\frac{1}{k}} =1,
\end{equation}
for $q>1$, and
for $q=1$, we write $d_q=d_1$ for the unique solution to
\begin{equation}\label{def_sad1}
\lim_{k\to \infty} \frac{1}{k}\int \log\Big(\psi^{d_1}(T_{\bu| k})^{-1} p_{\bu| k} \Big) d\mu(\bu) = 0.
\end{equation}

Given a  nonautonomous affine set $E^\w$ with random translations. If $
\Xi_{k}=\Xi$ for all $k>0 $,
the attractor $E^\omega$ is called  an {\em almost self-affine set}. Almost self-affine sets were first studied by Jordan,  Pollicott and Simon in \cite{JPS07} for dimension theory. Later, Falconer~\cite{Falco10} studied generalized $q$-dimension of measures supported on such sets, and he provided  generalized $q$-dimension  formula for $q>1$. In the following conclusion we improve it to $q\geq 1$.

\begin{cor}\label{Dq_aa}
Let $E^\w$ be the nonautonomous affine set with random translation, where $\Xi_k=\Xi=\{T_{1},T_{2},\ldots,T_{n_0}\}$.  Let $\mu^\w$ be the nonautonomous affine measure with  $\bp_k=\bp=(p_{1},p_{2},\cdots, p_{n_0})$ for all $k>0$. Then for all $q>0$, $d_q^-=d_q^+=d_q$. %
For $q \geq 1$, the generalized $q$-dimension $D_q(\mu^\w)$ exists and
$$
D_q(\mu^\w)=\min\{d_q,d\},
$$
for almost all $\w$, where $d_q$ is given by \eqref{def_sadg} and \eqref{def_sad1}.
\end{cor}

Finally, we consider a special case studied in \cite{GM22},  where the translations of affine mappings in the nonautonomous structure are selected only from a finite set.
Let $\Gamma=\{a_1,\ldots, a_\tau\}$ be a finite collection of translations, where $a_1,\ldots, a_\tau$ are regarded later as variables in $\R^d$. For each $\mathbf{u}=u_1\ldots u_k \in \Sigma^*$,  we have that $J_{\mathbf{u}}=\Psi_{\mathbf{u}}(J)=\Psi_{u_1}\circ  \ldots \circ \Psi_{u_k} (J)$. Suppose that the translation of $\Psi_{u_j}$ is an element of $\Gamma$, that is,
\begin{equation}\label{mpsac}
\Psi_{u_j}(x)=T_{j,u_j}x+ \w_{u_1\ldots u_j}, \qquad \w_{u_1\ldots u_j}\in \Gamma,
\end{equation}
 for $j=1,2,\ldots, k$.
  We write $\ba=(a_1,\ldots, a_\tau)$ as a variable in $\R^{\tau d}$. To emphasize the dependence on these special translations in $\Gamma$,  we denote the nonautonomous affine set by $E^\ba$.

\iffalse
\begin{thm}\label{dimmua}
Given $\Gamma=\{a_1,\ldots, a_\tau\}$. Let $E^\ba$ be the nonautonomous affine set given by ~\eqref{attractor} where the translations of affine mappings are chosen from $\Gamma$.  Suppose that
$$
\sup\{\|T_{k,j}\| : 0<j\leq n_k,  k>0 \}< \frac{1}{2}.
$$
Let $\mu^\ba$ be the nonautonomous affine measure with  $\bp_k=\bp=(p_{1},p_{2},\cdots, p_{n_0})$ for all $k>0$.
Then for $1<q\leq2$, we have that
$$
\underline{D}_q(\mu^\ba)=\min\{d_q^-,d\},
$$
for  $\mathcal{L}^{\tau d}$-almost all $\ba\in \R^{\tau d}$, where
$$
d_q^-=\sup \Big\{s: \sum_{k=1}^\infty \sum_{\bu\in\Sigma^k} \psi^s(T_\bu)^{1-q}p_\bu^q<\infty\Big\}.
$$
\end{thm}
\fi

\begin{thm}\label{dimmua}
Given $\Gamma=\{a_1,\ldots, a_\tau\}$. Let $E^\ba$ be the nonautonomous affine set given by ~\eqref{attractor} and \eqref{mpsac}. Let $\mu^\ba$ be the nonautonomous affine measure.
 Suppose that
$$
\sup\{\|T_{k,j}\| : 0<j\leq n_k,  k>0 \}< \frac{1}{2}.
$$
Then for $1<q\leq2$, we have that
$$
\underline{D}_q(\mu^\ba)=\min\{d_q^-,d\},
$$
for  $\mathcal{L}^{\tau d}$-almost all $\ba\in \R^{\tau d}$, where
$$
d_q^-=\sup \Big\{s: \sum_{k=1}^\infty \sum_{\bu\in\Sigma^k} \psi^s(T_\bu)^{1-q}\mu(\mathcal{C}_\bu)^q<\infty\Big\}.
$$
\end{thm}

Similar to almost self-afffine sets, supppose that $E^\ba$  is the nonautonomous affine set generated by identical $\Xi_{k}=\Xi=\{T_{1},T_{2},\ldots,T_{n_0}\}$ for all $k>0$ with the translations  chosen from $\Gamma$. Note that $E^\ba$  is not necessarily a self-affine set, and the self-affine set generated by IFS
 $$
\{\Psi_{j}(x)=T_{j}x+ a_{j}\}_{j=1}^{n_0}, \qquad a_{j}\in \Gamma,
  $$
is a special case of $E^\ba$. For the nonautonomous affine measure $\mu^\ba$  with $\bp_k=\bp=(p_{1},p_{2},\cdots, p_{n_0})$ for all $k>0$ supported on $E^\ba$, the generalized $q$-dimension of $\mu^\ba$ exists for $1\leq  q\leq 2$.

\begin{cor}\label{cor_nasid}
Given $\Gamma=\{a_1,\ldots, a_\tau\}$. Let $E^\ba$ be the nonautonomous affine set given by ~\eqref{attractor} and \eqref{mpsac} with $\Xi_{k}=\Xi=\{T_{1},T_{2},\ldots,T_{n_0}\}$ for all $k>0$.
Let $\mu^\ba$ be the nonautonomous affine measure with  $\bp_k=\bp=(p_{1},p_{2},\cdots, p_{n_0})$ for all $k>0$. %
Suppose that
$$
\sup\{\|T_j\| : 0<j\leq n_0 \}< \frac{1}{2}.
$$
Then for all $q>0$, $d_q^-=d_q^+=d_q$. %
For $1 \leq q \leq 2$, the generalized $q$-dimension $D_q(\mu^\ba)$ exists and
$$
D_q(\mu^\ba)=\min\{d_q,d\},
$$
for almost all $\ba$, where $d_q$ is  given by \eqref{def_sadg} and \eqref{def_sad1}.
\end{cor}
The generalized $q$-dimension of self-affine measures was first studied  by Falconer in \cite{Falco05}, and he proved the generalized $q$-dimension formula for  $1< q\leq 2$. As  an immediate consequence of Corollary \ref{cor_nasid}, we improve it $1\leq  q\leq 2$.
\begin{cor}
Let $E^\ba$ be the self-affine set given by  $\Xi=\{T_{1},T_{2},\ldots,T_{n_0}\}$.
Let $\mu^\ba$ be the self-affine measure with  $\bp=(p_{1},p_{2},\cdots, p_{n_0})$. %
Suppose that
$$
\sup\{\|T_j\| : 0<j\leq n_0 \}< \frac{1}{2}.
$$
Then for all $q>0$, $d_q^-=d_q^+=d_q$. %
For $1 \leq q \leq 2$, the generalized $q$-dimension $D_q(\mu^\ba)$ exists and
$$
D_q(\mu^\ba)=\min\{d_q,d\},
$$
for almost all $\ba$, where $d_q$ is given by \eqref{def_sadg} and \eqref{def_sad1}.
\end{cor}

\section{Generalized $q$-dimensions of measures on Non-autonomous fractals}\label{sec 3}
\setcounter{equation}{0}
\setcounter{pro}{0}

In this section, we study the generalized $q$-dimensions of measures supported on nonautonomous fractals, and we always assume that
$$
c_*=\inf_{k,j} \{c_{k,j}\} > 0, \qquad\qquad c^*=\sup_{k,j} \{c_{k,j}\} < 1.
$$
Without loss of generality, we always assume that $|J|=1$.

First we given the upper bounds for the upper and lower  generalized $q$-dimensions.

\begin{proof}[Proof of Theorem~\ref{thm_DqU}]
We only give the proof for $\underline{D}_q(\mu^\w)\leq \min\{d_q^-,d\}$ since the proof for $\overline{D}_q(\mu^\w)\leq\min\{d_q^-,d\}$ is almost identical.

Recall that for each $k>0$, the ratios $\{c_{k,i}\}$ are given by
$$
|S_{k,i}(x)-S_{k,i}(y)| \leq c_{k,i}|x-y|, \qquad \textit{ for }x,y \in J.
$$
It is sufficient to prove that
$$
\underline{D}_q(\mu^\w)\leq d_q^-,
$$
where $d_q^-$ is given by \eqref{def_dq-} with respect to $(c_{k,1},c_{k,2},\cdots, c_{k,n_{k}})$.

Let $\mu_\bu$ denote the restriction of $\mu$ to the cylinder $\mathcal{C}_\bu$, and let $\mu_\bu^\w$ be the image measure of $\mu_\bu$ under $\Pi^\w$. It is clear that the support of $\mu_\bu^\w$ is contained $J_\bu$, i.e.,  $\spt\ \mu _\bu^\w \subset J_\bu$, and 
$$
\mu(\mathcal{C}_\bu) =\mu_\bu(\mathcal{C}_\bu)= \mu_\bu^\w(J_\bu).
$$
Given reals $s>0$ and $c_*>r>0$, recall that $\Sigma^*(s,r)=\{\bu\in\Sigma^*: c_\bu \leq r < c_{\bu^-}\}$ is a cut set. For each $\bu \in \Sigma^*(s,r)$, we have
\begin{equation}\label{leqr}
  c_* r \leq |J_\bu| \leq r.
\end{equation}
Then $J_\bu$ intersects at most $2^d$ $r$-cubes in $\R^d$.

According to the definition of $d_q^-$, the proof is divided into $3$ parts: $q>1$, $1>q>0$ and $q=1$.

For $q>1$. For each $\bu\in \Sigma^*$, since $J_\bu$ intersects at most $2^d$ $r$-cubes in $\R^d$, by Jensen's inequality, it follows that
\begin{equation}\label{muwql}
\mu(\mathcal{C}_\bu)^q= \mu_\bu^\w(J_\bu)^q \leq C_1 \sum_{Q\in\mathcal{M}_r} \mu_\bu^\w(Q)^q,
\end{equation}
where $C_1$ is a constant that only depends on $d$. For each $Q\in \mathcal{M}_r$, by the power inequality, we have that
$$
\sum_{\bu\in\Sigma^*(s,r)} \mu_\bu^\w(Q)^q \leq \Big(\sum_{\bu\in\Sigma^*(s,r)} \mu_\bu^\w(Q)\Big)^q = \mu^\w(Q)^q.
$$
Summing up ~\eqref{muwql}, it follows  that
\begin{equation}\label{leqmuw}
\sum_{\bu\in\Sigma^*(s,r)} \mu(\mathcal{C}_\bu)^q \leq C_1  \sum_{Q\in\mathcal{M}_r} \mu^\w(Q)^q.
\end{equation}

For every $s>d_q^-$,  by \eqref{def_dq-}, we have that
$$
\limsup_{r\to 0}\sum_{\bu\in\Sigma^*(s, r)}c_\bu^{s(1-q)} \mu(\mathcal{C}_\bu)^q=\infty.
$$
Combining with ~\eqref{leqr} and ~\eqref{leqmuw}, this implies that
$$
s \geq \liminf_{r\to 0}\frac{\log\sum_{\bu\in\Sigma^*(s, r)}\mu(\mathcal{C}_\bu)^q}{(q-1)\log r}\geq \liminf_{r\to 0}\frac{\log \sum_{Q\in\mathcal{M}_r} \mu^\w(Q)^q}{(q-1)\log r}=\underline{D}_q(\mu^\w),
$$
and it immediately follows that
$$
\underline{D}_q(\mu^\w)\leq d_q^-.
$$

For $1>q>0$. By Jensen's inequality, we have that for each $J_\bu$,
\begin{equation}\label{muwqg}
\mu_\bu^\w(J_\bu)^q \geq C_2 \sum_{Q\in\mathcal{M}_r} \mu_\bu^\w(Q)^q.
\end{equation}
where $C_2$ is a constant that only depends on $d$. For each $Q\in \mathcal{M}_r$, by the power inequality, we have
$$
\sum_{\bu\in\Sigma^*(s,r)} \mu_\bu^\w(Q)^q \geq \Big(\sum_{\bu\in\Sigma^*(s,r)} \mu_\bu^\w(Q)\Big)^q = \mu^\w(Q)^q.
$$
Summing up ~\eqref{muwqg}, we obtain that
$$
\sum_{\bu\in\Sigma^*(s,r)} \mu_\bu^\w(J_\bu)^q \geq C_2  \sum_{Q\in\mathcal{M}_r} \mu^\w(Q)^q.
$$
Since $\mu_\bu^\w(J_\bu)=\mu(\mathcal{C}_\bu)$, it follows that
\begin{equation}\label{geqmuw}
\sum_{\bu\in\Sigma^*(s,r)}\mu(\mathcal{C}_\bu)^q \geq C_2  \sum_{Q\in\mathcal{M}_r} \mu^\w(Q)^q.
\end{equation}

For every $s>d_q^-$, by \eqref{def_dq-}, we have that
$$
\liminf_{r\to 0}\sum_{\bu\in\Sigma^*(s, r)}c_\bu^{s(1-q)}\mu(\mathcal{C}_\bu)^q<\infty.
$$
Combining with ~\eqref{leqr} and ~\eqref{geqmuw}, this implies that
$$
s \geq \liminf_{r\to 0}\frac{\log\sum_{\bu\in\Sigma^*(s, r)}\mu(\mathcal{C}_\bu)^q}{(q-1)\log r}\geq \liminf_{r\to 0}\frac{\log \sum_{Q\in\mathcal{M}_r} \mu^\w(Q)^q}{(q-1)\log r}=\underline{D}_q(\mu^\w),
$$
and it immediately follows that
$$
\underline{D}_q(\mu^\w)\leq d_q^-.
$$

For $q=1$. Since $\mu(\mathcal{C}_\bu) = \mu_\bu^\w(J_\bu) $ and  $J_\bu$ intersects at most $2^d$ $r$-cubes in $\R^d$, by Jensen's inequality, we have that for each $J_\bu$,
\begin{equation}\label{muwlg}
\mu(\mathcal{C}_\bu) \log \mu(\mathcal{C}_\bu)\leq \sum_{Q\in\mathcal{M}_r} \mu_\bu^\w(Q) \log \mu_\bu^\w(Q) + \sum_{Q\in\mathcal{M}_r} \mu_\bu^\w(Q) \log 2^d.
\end{equation}
where $C_3$ is a constant that only depends on $d$. For each $Q\in \mathcal{M}_r$ we have
$$
\sum_{\bu\in\Sigma^*(s,r)} \mu_\bu^\w(Q) \log \mu_\bu^\w(Q) \leq \sum_{\bu\in\Sigma^*(s,r)} \mu_\bu^\w(Q) \log \Big(\sum_{\bu\in\Sigma^*(s,r)} \mu_\bu^\w(Q)\Big) = \mu^\w(Q) \log \mu^\w(Q).
$$
Summing up ~\eqref{muwlg}, we obtain that
\begin{equation}\label{logmuw}
\sum_{\bu\in\Sigma^*(s,r)}\mu(\mathcal{C}_\bu) \log \mu(\mathcal{C}_\bu)\leq \sum_{Q\in\mathcal{M}_r} \mu^\w(Q) \log \mu^\w(Q) + \log 2^d.
\end{equation}

For every $s>d_1^-$,  by \eqref{def_dq-}, we have that
$$
\limsup_{r\to 0}\sum_{\bu\in\Sigma^*(s, r)}\mu(\mathcal{C}_\bu)\log\Big(c_\bu^{-s}\mu(\mathcal{C}_\bu)\Big)=\infty.
$$
Combining with ~\eqref{leqr} and ~\eqref{logmuw}, this implies that
\begin{eqnarray*}
s &\geq& \liminf_{r\to 0}\frac{\sum_{\bu\in\Sigma^*(s, r)} \mu(\mathcal{C}_\bu)\log \mu(\mathcal{C}_\bu) }{\log r}              \\
&\geq& \liminf_{r\to 0}\frac{ \sum_{Q\in\mathcal{M}_r} \mu^\w(Q) \log \mu^\w(Q)}{\log r}           \\ &=&\underline{D}_1(\mu^\w),
\end{eqnarray*}
and it immediately follows that
$$
\underline{D}_1(\mu^\w)\leq d_1^-.
$$

Since $\underline{D}_q(\mu^\w)\leq d_q^-$  holds for all $q>0$,  we obtain that
$$
\underline{D}_q(\mu^\w)\leq\min\{d_q^-,d\},
$$
and we complete the proof.
\end{proof}

Next, we give the lower bounds for the upper and lower  generalized $q$-dimensions of measures supported on nonautonomous attractors.

\begin{proof}[Proof of Theorem~\ref{thm_DqL}]
We only give the proof for $\underline{D}_q(\mu^\w)\geq d_q^-$ since the other proof  is similar.

Given $s>0$ and $c_*>r>0$. Since $E$ satisfies gap separation condition, for all $\bu\neq\bv\in \Sigma^*(s,r)$, we have that
$$\inf\{|x-y|:x\in J_\bu,y\in J_\bv\} \geq C|J_{\bu^-}| \geq
C c_{\bu^-} > C r,$$
and it implies that every $Q\in M_r$ intersects at most  $N$ elements in $\{J_\bu: \bu \in \Sigma^*(s,r)\}$, where $N$ is a constant that only depends on $C$.
% Then by similar argument to above, we obtain the conclusion.

For $q>1$. By Jensen's inequality, we have that for each $Q\in \mathcal{M}(r)$,
\begin{equation}\label{muwqq}
\mu^\w(Q)^q \leq C_1 \sum_{\bu \in \Sigma^*(s,r)} \mu^\w(Q \cap J_\bu)^q,
\end{equation}
where $C_1$ is a constant that only depends on $N$. For each $J_\bu$, by the power inequality, we have
$$
\mu^\w(J_\bu)^q \geq \sum_{Q\in \mathcal{M}(r)} \mu^\w(Q\cap J_\bu).
$$
Summing up ~\eqref{muwqq}, we obtain that
\begin{equation}\label{leqmuwq}
\sum_{Q\in\mathcal{M}_r} \mu^\w(Q)^q \leq C_1  \sum_{\bu\in\Sigma^*(s,r)} \mu^\w(J_\bu)^q.
\end{equation}

For each $s<d_q^-$, by \eqref{def_dq-}, we have that
$$
\limsup_{r\to 0}\sum_{\bu\in\Sigma^*(s, r)}c_\bu^{s(1-q)}\mu(\mathcal{C}_\bu)^q<\infty.
$$
Since $\mu^\w(J_\bu)= \mu(\mathcal{C}_\bu)$ for all $\bu\in\Sigma^*(s,r)$ and  $  c_* r \leq |J_\bu| \leq r$ for all $\bu \in \Sigma^*(s,r)$, combining with  ~\eqref{leqmuwq}, it implies that
$$
s \leq \liminf_{r\to 0}\frac{\log\sum_{\bu\in\Sigma^*(s, r)}\mu^\w(J_\bu)^q}{(q-1)\log r}\leq \liminf_{r\to 0}\frac{\log \sum_{Q\in\mathcal{M}_r} \mu^\w(Q)^q}{(q-1)\log r}=\underline{D}_q(\mu^\w).
$$

Hence for $q>1$, we have that
$$\underline{D}_q (\mu^\w)\geq d_q^-.$$
The proofs for $0<q<1$ and $q=1$ are similar, and we omit it.
\end{proof}

Recall that $E$ is a nonautonomous similar set, that is, with
$$
|S_{k,i}(x)-S_{k,i}(y)| =  c_{k,i}|x-y|, \qquad  \textit{ for }x,y \in J,
$$
where $c_{k,i}<1$,  for all $i=1,2,\ldots, n_k$ and for all $ k>0$.
Next, we study the  generalized $q$-dimensions of measures supported on  nonautonomous similar set, and show that upper and lower generalized $q$-dimensions are given by $d_q^+$ and $d_q^-$ respectively under strong separation condition.

\begin{proof}[Proof of Theorem~\ref{thm_MS}]
We only give the proof for  $\underline{D}_q(\mu^\w)=\min\{ d_q^-,d\}$ since the other proof  is similar.  By Theorem \ref{thm_DqU}, it immediately follows that
$$
\underline{D}_q(\mu^\w)\leq \min\{ d_q^-,d\},
$$
and it remains to show that $\underline{D}_q(\mu^\w) \geq d_q^-$.

Given $s>0$ and $c_*>r>0$. For each $\bu \in \Sigma^*(s,r)$, the fact $ c_* r < |J_\bu| \leq r$ implies that the basic set $J_\bu$ intersects at most $2^d$ $r$-cubes in $\R^d$.
On the other hand,
each $r$-cube intersects at most $N$ basic sets in $\{J_\bu: \bu \in \Sigma^*(s,r)\}$, where the constant $N$ only depends on $J$ and $c_*$.

For $q>1$. By Jensen's inequality, we have that for each $J_\bu$,
\begin{equation}\label{vgq}
\mu^\w(J_\bu)^q \leq C \sum_{Q\in\mathcal{M}_r} \mu^\w(Q \cap J_\bu)^q,
\end{equation}
and for each $Q\in \mathcal{M}_r$,
\begin{equation}\label{vqq}
\mu^\w(Q)^q \leq C' \sum_{\bu\in\Sigma^*(s,r)} \mu^\w(J_\bu \cap Q)^q,
\end{equation}
where $C$ and $C'$ are constants that only depends on $J$ and $c_*$. Summing up ~\eqref{vgq} and ~\eqref{vqq} respectively, we obtain that
$$
\sum_{\bu\in\Sigma^*(s,r)} \mu^\w(J_\bu)^q \leq C N \sum_{Q\in\mathcal{M}_r} \mu^\w(Q)^q,
$$
and
$$
\sum_{Q\in\mathcal{M}_r} \mu^\w(Q)^q \leq C' 2^d \sum_{\bu\in\Sigma^*(s,r)} \mu^\w(J_\bu)^q.
$$
By taking $C_1=\max\{C N, C' 2^d\}$, it follows that
\begin{equation}\label{leqmu}
C_1^{-1} \sum_{Q\in \mathcal{M}_r} \mu^\w(Q)^q \leq \sum_{\bu \in \Sigma^*(s,r)} (\mu^\w(J_\bu))^q \leq C_1 \sum_{Q\in \mathcal{M}_r} \mu^\w(Q)^q.
\end{equation}

For all $s<d_q^-$,   by \eqref{def_dq-}, we have that
$$
\limsup_{r\to 0}\sum_{\bu\in\Sigma^*(s, r)}c_\bu^{s(1-q)} \mu(\mathcal{C}_\bu)^q<\infty.
$$
Since $E$ satisfies SSC, we have that  $\mu^\w(J_\bu)=\mu(\mathcal{C}_\bu)$ for all $\bu\in\Sigma^*(s, r),$ and this implies that
$$
\limsup_{r\to 0}\sum_{\bu\in\Sigma^*(s, r)}c_\bu^{s(1-q)} \mu^\w(J_\bu)^q<\infty.
$$
Combining with ~\eqref{leqmu}, this implies that
$$
s \leq \liminf_{r\to 0}\frac{\log\sum_{\bu\in\Sigma^*(s, r)}\mu^\w(J_\bu)^q}{(q-1)\log r}\leq \liminf_{r\to 0}\frac{\log \sum_{Q\in\mathcal{M}_r} \mu^\w(Q)^q}{(q-1)\log r}=\underline{D}_q(\mu^\w),
$$
and it immediately follows that
$$
\underline{D}_q(\mu^\w)\geq d_q^-.
$$
The proofs for $0<q<1$ and $q=1$ are similar, and we omit it.
\end{proof}

The critical values $d_q^-$ and $d_q^+$ are key to study the generalized $q$-dimensions. If the measure $\mu$ is a Bernoulli measure on $\Sigma^\infty$,  then we may further simplify the form of  $d_q^-$ and $d_q^+$ for  nonautonomous similar sets.

\begin{proof}[Proof of Proposition~\ref{leq_dq}]
Since $\mu$ is a Bernoulli measure given by $\{\bp_k\}$, we have
$$
\mu(\mathcal{C}_\bu)=p_\bu=p_{1,u_1}\ldots p_{k,u_k}.
$$

Given $s>0$, $c_*>r>0$. Let
$$
K_1=\max\{|\bu|: \bu\in \Sigma^*(s,r)\},\qquad K_2=\min\{|\bu|: \bu\in \Sigma^*(s,r)\}.
$$
It is clear that  $K_1 \geq K_2 \geq 1$.

For each $\bu\in \Sigma^*(s,r)$ with length $K_1$, it is clear that $\bu^- j\in \Sigma^*(s,r)$ for all $j\in\{1,\ldots,n_{K_1}\}$. If $\sum_j c_{K_1,j}^{s(1-q)} p_{K_1,j}^q \geq 1$, then
$$
\sum_j c_{\bu^- j}^{s(1-q)} p_{\bu^- j}^q \geq c_{\bu^-}^{s(1-q)} p_{\bu^-}^q \Big(\sum_j c_{K_1,j}^{s(1-q)} p_{K_1,j}^q \Big) \geq c_{\bu^-}^{s(1-q)} p_{\bu^-}^q.
$$
If $\sum_j c_{K_1,j}^{s(1-q)} p_{K_1,j} < 1$, then
\begin{eqnarray*}
&&\hspace{-1cm}\! \! \sum_{\bu\in \Sigma^*(s,r) \atop |\bu|=K_1-1} c_{\bu}^{s(1-q)} p_{\bu}^q + \! \! \sum_{\bu\in \Sigma^*(s,r) \atop |\bu|=K_1} c_{\bu}^{s(1-q)} p_{\bu}^q \geq \! \!\! \!\sum_{\bu\in \Sigma^*(s,r) \atop |\bu|=K_1-1} c_{\bu}^{s(1-q)} p_{\bu}^q\Big(\sum_j c_{K_1,j}^{s(1-q)} p_{K_1,j}^q\Big)+\! \! \! \!\sum_{\bu\in \Sigma^*(s,r) \atop |\bu|=K_1} c_{\bu}^{s(1-q)} p_{\bu}^q \\
&&\hspace{5cm} \geq \sum_{\bu\in A_1} c_{\bu}^{s(1-q)} p_{\bu}^q \Big(\sum_j c_{K_1,j}^{s(1-q)} p_{K_1,j}^q \Big).
\end{eqnarray*}
where $A_1=\{\bu\in \Sigma^*(s,r):|\bu|=K_1-1\}\cup \{\bu^-:\bu\in \Sigma^*(s,r), |\bu|=K_1\}$.

%重复上述步骤

By repeating this process, we have
$$
\sum_{\bu\in \Sigma^*(s,r)} c_\bu^{s(1-q)} p_\bu^q \geq \sum_{\bu\in \Sigma^k} c_\bu^{s(1-q)} p_\bu^q  =\prod_{i=1}^k \sum_{j=1}^{n_i}c_{i,j}^{s(1-q)}p_{i,j}^q.
$$
for some $K_2 \leq k \leq K_1$, and it implies that
$$
\liminf_{r\to 0}\sum_{\bu\in\Sigma^*(s, r)}c_\bu^{s(1-q)}\mu(\mathcal{C}_\bu)^q \geq \liminf_{k\to \infty}\prod_{i=1}^k \sum_{j=1}^{n_i}c_{i,j}^{s(1-q)}p_{i,j}^q.
$$
It follows that for $q>1$,
$$
d_q^+ \leq \sup \Big\{s: \liminf_{k\to \infty} \prod_{i=1}^k \sum_{j=1}^{n_i}c_{i,j}^{s(1-q)}p_{i,j}^q<\infty\Big\}.
$$
and for $0<q<1$,
$$
d_q^- \geq \inf \Big\{s: \liminf_{k\to \infty} \prod_{i=1}^k \sum_{j=1}^{n_i}c_{i,j}^{s(1-q)}p_{i,j}^q<\infty\Big\}.
$$
Hence ~\eqref{dq+leq1} and ~\eqref{dq-geq1} hold.

For $q=1$, for each $\bu\in \Sigma^*(s,r)$ with length $K_1$, it is clear that $\bu^- j\in \Sigma^*(s,r)$ for all $j\in\{1,\ldots,n_{K_1}\}$. If $\sum_j p_{K_1,j} \log \left(c_{K_1,j}^{-s} p_{K_1,j}\right) \geq 0$, then
\begin{eqnarray*}
\sum_j p_{\bu^- j} \log \Big(c_{\bu^- j}^{-s} p_{\bu^- j}\Big) &=& p_{\bu^-} \log \left(c_{\bu^-}^{-s} p_{\bu^-}\right)+ p_{\bu^-} \Big(\sum_j p_{K_1,j} \log\left(c_{K_1,j}^{-s} p_{K_1,j}\right)\Big) \\
&\geq& p_{\bu^-} \log \Big(c_{\bu^-}^{-s} p_{\bu^-}\Big).
\end{eqnarray*}
If $\sum_j p_{K_1,j} \log (c_{K_1,j}^{-s} p_{K_1,j}) < 0$, then
\begin{eqnarray*}
&&\sum_{\bu\in \Sigma^*(s,r) \atop |\bu|=K_1-1} p_{\bu} \log \left(c_{\bu}^{-s} p_{\bu}\right) + \sum_{\bu\in \Sigma^*(s,r) \atop |\bu|=K_1} p_{\bu} \log \left(c_{\bu}^{-s} p_{\bu}\right) \\
&=& \sum_{\bu\in \Sigma^*(s,r) \atop |\bu|=K_1-1} p_{\bu} \log \left(c_{\bu}^{-s} p_{\bu}\right) + \sum_{\bu\in \Sigma^*(s,r) \atop |\bu|=K_1} \Big[ p_{\bu^-} \log \left(c_{\bu^-}^{-s} p_{\bu^-}\right)+ p_{\bu^-} \Big(\sum_j p_{K_1,j} \log\left(c_{K_1,j}^{-s} p_{K_1,j}\right)\Big) \Big] \\
&\geq& \sum_{\bu\in A_1} \Big[p_{\bu} \log \left(c_{\bu}^{-s} p_{\bu}\right) + p_{\bu} \Big(\sum_j p_{K_1,j} \log\left(c_{K_1,j}^{-s} p_{K_1,j}\right) \Big) \Big].
\end{eqnarray*}
where $A_1=\{\bu\in \Sigma^*(s,r):|\bu|=K_1-1\}\cup \{\bu^-:\bu\in \Sigma^*(s,r), |\bu|=K_1\}$.

%重复上述步骤

By repeating this process, we have
$$
\sum_{\bu\in \Sigma^*(s,r)}p_{\bu} \log \left(c_{\bu}^{-s} p_{\bu}\right) \geq \sum_{\bu\in \Sigma^k} p_{\bu} \log \left(c_{\bu}^{-s} p_{\bu}\right).
$$
for some $K_2 \leq k \leq K_1$, and thus
$$
\liminf_{r\to 0}\sum_{\bu\in\Sigma^*(s, r)}p_{\bu} \log \left(c_{\bu}^{-s} p_{\bu}\right) \geq \liminf_{k\to \infty} \sum_{\bu\in \Sigma^k} p_{\bu} \log \left(c_{\bu}^{-s} p_{\bu}\right).
$$
By the same argument, we have
$$
\limsup_{r\to 0}\sum_{\bu\in\Sigma^*(s, r)}p_{\bu} \log \left(c_{\bu}^{-s} p_{\bu}\right) \leq \limsup_{k\to \infty} \sum_{\bu\in \Sigma^k} p_{\bu} \log \left(c_{\bu}^{-s} p_{\bu}\right).
$$
Since
$$
\sum_{\bu\in \Sigma^k} p_{\bu} \log \left(c_{\bu}^{-s} p_{\bu}\right) =\sum_{i=1}^k \sum_{j=1}^{n_i} \big(p_{i,j}\log p_{i,j} - s\ p_{i,j}\log c_{i,j}\big)
$$
the inequalities ~\eqref{dq-geq2} and ~\eqref{dq+leq2} follow immediately.

Finally, we prove ~\eqref{dq-eq} and ~\eqref{dq+eq}.  We write
\begin{eqnarray}
d_1 &=& \sup \Big\{s: \limsup_{k\to \infty}\sum_{\bu\in\Sigma^k} c_\bu^{s(1-q)}p_\bu^q<\infty\Big\}, \\
d_2 &=& \sup \Big\{s: \sum_{k=1}^\infty \sum_{\bu\in\Sigma^k} c_\bu^{s(1-q)}p_\bu^q<\infty\Big\}.
\end{eqnarray}
It is sufficient to show that $d_q^-=d_1=d_2$.

Note that, for every $|\bu|=k$ and $0<s_1<s$, the following inequality holds
\begin{equation}\label{phis1}
c_\bu^{s_1}= c_\bu^{(s_1-s)} c_\bu^s \geq (c^*)^{-k(s-s_1)}c_\bu^s.
\end{equation}
For each non-integral  $s<d_1$, we write that
$$
\limsup_{k\to \infty} \sum_{\bu\in\Sigma^k} c_\bu^{s(1-q)}p_\bu^q=M<\infty.
$$
It  immediately follows that for all $s_1<s$,
$$
\sum_{k=1}^\infty \sum_{\bu\in\Sigma^k} c_\bu^{s_1(1-q)} p_\bu^q \leq \sum_{k=0}^\infty M (c^*)^{k(s-s_1)(q-1)}<\infty,
$$
and it implies that $d_1\leq d_2$.

For each $s<d_2$, we have that
$$
\sum_{k=1}^\infty \sum_{\bu\in\Sigma^k} c_\bu^{s(1-q)}p_\bu^q<\infty.
$$
Since
$$
\sum_{\bu\in\Sigma^*(s, r)}c_\bu^{s(1-q)}p_\bu^q\leq \sum_{k=1}^\infty \sum_{\bu\in\Sigma^k} c_\bu^{s(1-q)}p_\bu^q,
$$
we obtain that $d_2\leq d_q^-$.

On the other hand, for all non-integral $s<d_q^-$, suppose that
$$
\limsup_{r\to 0}\sum_{\bu\in\Sigma^*(s, r)}c_\bu^{s(1-q)}p_\bu^q=M.
$$
For each $s_1<s$, For all  $\bu\in \Sigma^*(s,r)$,  recall that $ c_{*} r < c_\mathbf{u} \leq r
$,  and this implies that
$$
c_\bu^{s_1} \geq r^{s_1-s} c_\bu^s.
$$
It follows that
$$
\sum_{\bu\in\Sigma^*(s, r)}c_\bu^{s_1(1-q)}p_\bu^q\leq Mr^t.
$$
where $t=(s-s_1)(q-1)>0$.

By choosing $\rho$ such that $\alpha_+<\rho<1$, we have that
$$
\Sigma^*=\bigcup_{k=0}^\infty \Sigma^k=\bigcup_{l=0}^\infty\Sigma^*(s,\rho^l),
$$ and  it implies that
\begin{eqnarray*}
 \sum_{\bu\in\Sigma^k} c_\bu^{s_1(1-q)}p_\bu^q &\leq& \sum_{k=0}^\infty \sum_{\bu\in\Sigma^k} c_\bu^{s_1(1-q)}p_\bu^q   \\
 &\leq& \sum_{l=0}^\infty \sum_{\bu\in\Sigma^*(s,\rho^l)} c_\bu^{s_1(1-q)}p_\bu^q   \\
 &\leq& \sum_{l=0}^\infty M\rho^{lt},
\end{eqnarray*}
for all $k>0$. It follows that
$$
\limsup_{k\to \infty} \sum_{\bu\in\Sigma^k} c_\bu^{s_1(1-q)}p_\bu^q<\infty,
$$
and we obtain that $d_q^-\leq d_1$. Therefore we obtain that $d_1=d_2 = d_q^-$. Since
$$
 \sum_{\bu\in\Sigma^k} c_\bu^{s_1(1-q)}p_\bu^q= \prod_{i=1}^k \sum_{j=1}^{n_i}c_{i,j}^{s(1-q)}p_{i,j}^q
$$
the equality~\eqref{dq-eq}  holds.

To prove \eqref{dq+eq}, similarly we write
\begin{eqnarray}
d_3 &=& \inf \Big\{s: \limsup_{k\to \infty}\sum_{\bu\in\Sigma^k} c_\bu^{s(1-q)}p_\bu^q<\infty\Big\}, \\
d_4 &=& \inf \Big\{s: \sum_{k=1}^\infty \sum_{\bu\in\Sigma^k} c_\bu^{s(1-q)}p_\bu^q<\infty\Big\}.
\end{eqnarray}
It is sufficient to show that $d_q^+=d_3=d_4$.

Note that, for every $|\bu|=k$ and $0<s<s_1$, the following inequality holds
\begin{equation}\label{phis1}
c_\bu^{s_1}\leq c_\bu^{(s_1-s)}c_\bu^s \leq (c^*)^{k(s_1-s)}c^s_\bu.
\end{equation}
For each non-integral  $s>d_3$, we write that
$$
\limsup_{k\to \infty} \sum_{\bu\in\Sigma^k} c_\bu^{s(1-q)}p_\bu^q=M<\infty.
$$
It  immediately follows that for all $s_1>s$,
$$
\sum_{k=1}^\infty \sum_{\bu\in\Sigma^k}c_\bu^{s(1-q)}p_\bu^q \leq \sum_{k=0}^\infty M (c^*)^{k(s_1-s)(1-q)}<\infty,
$$
and it implies that $d_4\leq d_3$.

For each $s>d_4$, we have that
$$
\sum_{k=1}^\infty \sum_{\bu\in\Sigma^k} c_\bu^{s(1-q)}p_\bu^q<\infty.
$$
Since
$$
\sum_{\bu\in\Sigma^*(s, r)}c_\bu^{s(1-q)}p_\bu^q\leq \sum_{k=1}^\infty \sum_{\bu\in\Sigma^k} c_\bu^{s(1-q)}p_\bu^q,
$$
we obtain that $d_q^+\leq d_4$.

On the other hand, for all non-integral $s>d_q^+$, suppose that
$$
\limsup_{r\to 0}\sum_{\bu\in\Sigma^*(s, r)} c_\bu^{s(1-q)}p_\bu^q      =M.
$$
For each $s_1>s$,  we have  that $c_\bu^{s_1} \leq r^{s_1-s} c_\bu^s $ for all  $\bu\in \Sigma^*(s,r)$.
It follows that
$$
\sum_{\bu\in\Sigma^*(s, r)}c_\bu^{s_1(1-q)}p_\bu^q\leq Mr^t.
$$
where $t=(s_1-s)(1-q)>0$.

By choosing $\rho$ such that $\alpha_+<\rho<1$, we have that  $\Sigma^*=\bigcup_{k=0}^\infty \Sigma^k=\bigcup_{l=0}^\infty\Sigma^*(s,\rho^l)$, and  it implies that
\begin{eqnarray*}
\sum_{\bu\in\Sigma^k} c_\bu^{s_1(1-q)}p_\bu^q &\leq&\sum_{k=0}^\infty \sum_{\bu\in\Sigma^k} c_\bu^{s_1(1-q)}p_\bu^q  \\
&\leq& \sum_{l=0}^\infty \sum_{\bu\in\Sigma^*(s,\rho^l)} c_\bu^{s_1(1-q)}p_\bu^q   \\
&\leq& \sum_{l=0}^\infty M\rho^{lt},
\end{eqnarray*}
for all $k>0$. It follows that
$$
\limsup_{k\to \infty} \sum_{\bu\in\Sigma^k} c_\bu^{s_1(1-q)}p_\bu^q<\infty,
$$
and we obtain that $d_3\leq d_q^+$. Therefore $d_3=d_4=d_q^+$, and the equality \eqref{dq+eq} holds.

\end{proof}

\begin{proof}[Proof of Corollary~\ref{cor_MS}]
Given a nonautonomous similar set $E$. Let $\mu^\w$ be the nonautonomous similar measure  with $\bp_k=\bp=(p_{1},p_{2},\cdots, p_{n_0})$ for all $k>0$. By Corollary~\ref{corNAss}, it is sufficient to prove that $d_q^-=d_q^+=d_q$
where  $d_q$ is the unique solution to
$$
\sum_{i=1}^{n_0} c_i^{d_q(1-q)} p_i^q=1 \quad   \textit{  for $q\neq 1$  \qquad  and  \qquad } %for $q= 1,$  }
d_1=\frac{\sum_{i=1}^{n_0} p_i \log p_i }{\sum_{i=1}^{n_0} p_i \log c_i }.
$$

For $q>0$ and $q\neq 1$. Since $E$ is a nonautonomous similar set with  $\Xi_{k}=\Xi=(c_{1},c_{2},\cdots, c_{n_0})$ for all $k>0$, and $\mu$ is a Benoulli probability measure on $\Sigma^\infty$ with $\bp_k=(p_{1},p_{2},\cdots, p_{n_0})$ for all  $k>0$, it is clear that
$$\lim_{k\to \infty} \Big(\sum_{\bu\in\Sigma^k} c_\bu^{s(1-q)} p_\bu^q\Big)^{\frac{1}{k}}=\sum_{i=1}^{n_0} c_i^{s(1-q)} p_i^q.
$$
This implies that
$$
d_q = \sup\Big\{s: \limsup_{k\to\infty} \sum_{\bu\in\Sigma^k} c_\bu^{s(1-q)} p_\bu^q<\infty\Big\}
  = \sup\Big\{s: \liminf_{k\to\infty} \sum_{\bu\in\Sigma^k} c_\bu^{s(1-q)} p_\bu^q<\infty\Big\},
$$
for $q>1$, and
$$
d_q = \inf\Big\{s: \liminf_{k\to\infty} \sum_{\bu\in\Sigma^k} c_\bu^{s(1-q)} p_\bu^q<\infty\Big\}
  = \inf\Big\{s: \limsup_{k\to\infty} \sum_{\bu\in\Sigma^k} c_\bu^{s(1-q)} p_\bu^q<\infty\Big\},
$$
for $0<q<1$.
By Proposition~\ref{leq_dq}, it immediately follows that
$$
d_q \leq d_q^- \leq d_q^+ \leq d_q,
$$
for all $q>0$ and $q\neq 1$, and we obtain that for all $q>0$ and $q\neq 1$,
$$
d_q^-=d_q^+=d_q.
$$

For $q=1$, it is clear that
$$
\sum_{i=1}^{n_0} p_i \log (c_i^{-d_1} p_i) = 0,
$$
and
$$
\lim_{k\to \infty} \frac{1}{k} \sum_{\bu\in \Sigma^k} p_\bu \log (c_\bu^{-s} p_\bu) = \sum_{i=1}^{n_0} p_i \log (c_i^{-s} p_i).
$$
Hence, we have that
$$
d_1 = \sup\big\{s: \limsup_{k\to\infty} \sum_{\bu\in \Sigma^k} p_\bu \log (c_\bu^{-s} p_\bu)<\infty\big\} = \sup\big\{s: \liminf_{k\to\infty} \sum_{\bu\in \Sigma^k} p_\bu \log (c_\bu^{-s} p_\bu)<\infty\big\}.
$$
Since
$$
\sum_{\bu\in \Sigma^k} p_\bu \log (c_\bu^{-s} p_\bu)= k \sum_{i=1}^{n_0} \Big(p_i\log p_i  -s p_i\log c_i \Big),
$$
by Proposition~\ref{leq_dq}, we have that $d_1 \leq d_1^- \leq d_1^+ \leq d_1$, and it immediately follows that
$$
d_1^-=d_1^+=d_1.
$$

Therefore, we obtain that $d_q^-=d_q^+=d_q$ for all $q>0$, and the conclusion holds.
\end{proof}

\section{Generalized $q$-dimensions of measures on nonautonomous affine sets}\label{sec 4}
\setcounter{equation}{0}
\setcounter{pro}{0}
In this section, we investigate upper bounds for generalized $q$-dimensions of measures supported on nonautonomous affine sets, and  it is often convenient to express generalized $q$-dimensions as integrals of measures of balls rather than as moment sums.

\begin{pro} \label{pro_eqdefDq}
The generalized $q$-dimensions have integral forms:
\be\label{gendimint}
\underline{D}_q(\nu)=\liminf_{r\to 0}\frac{\log\int\nu(B(x,r))^{q-1} d\nu(x)}{(q-1)\log r},
\overline{D}_q(\nu)=\limsup_{r\to 0}\frac{\log\int\nu(B(x,r))^{q-1} d\nu(x)}{(q-1)\log r},
\ee
for $q>0, q \neq 1$, and
\be\label{infdimint}
\underline{D}_1(\nu)=\liminf_{r\to 0}\frac{\int\log\nu(B(x,r)) d\nu(x)}{\log r}, \quad
\overline{D}_1(\nu)=\limsup_{r\to 0}\frac{\int\log\nu(B(x,r)) d\nu(x)}{\log r}.
\ee
Moreover, $\underline{D}_q(\nu)$ and $\overline{D}_q(\nu)$ are monotonic decreasing in $q$.

\end{pro}

\begin{proof}
Identity (\ref{gendimint}) is straightforward for $q>1$, see, for example, \cite{P}. The case of $0<q<1$ was established in \cite{PSol1}.

For $q=1$, given $0<r<1$, for each $x\in \mathbb{R}^d$ write $Q(x)$ for the $r$-mesh cube containing $x$. Then
$$
\sum_{\mathcal{M}_r}\nu(Q)\log\nu(Q)
= \int\log\nu(Q(x)) d\nu(x)
\leq \int\log\nu(B(x,\sqrt{d} r) d\nu(x),
$$
and dividing by $\log r$ and taking the limits gives that the expressions of (\ref{infdim}) are  at least the corresponding ones of  (\ref{infdimint}).

For the opposite inequalities, fix  $0<r<1$ and write $\widetilde{Q}$ for the cube of side $3r$ formed by the $3^N$ cubes in $\mathcal{M}_r$ consisting of $Q$ and its immediate neighbours.
Let $\mathcal{S}_k \; (k=1,2,3,\ldots)$ be the set of mesh cubes satisfying
$$\mathcal{S}_k = \{Q \in \mathcal{M}_r: 2^{k-1}\nu(Q) \leq \nu(\widetilde{Q} )< 2^{k}\nu(Q)\}.
$$
This implies that
\be
\sum_{Q \in \mathcal{S}_k}\nu(Q)
\leq 2^{1-k} \sum_{Q \in \mathcal{S}_k}\nu(\widetilde{Q})
\leq 2^{1-k} 3^d \sum_{Q \in  \mathcal{M}_r}\nu(Q) = 2^{1-k} 3^d. \label{cksize}
\ee
By \eqref{cksize}, it follows that
\begin{eqnarray*}
\int\log\nu(B(x,r)) d\nu(x)
&\leq&  \sum_{Q \in \mathcal{M}_r}\nu(Q) \log\nu(\widetilde{Q})\\
&\leq& \sum_{k=1}^\infty \sum_{Q \in \mathcal{S}_k}\nu(Q) \log(2^k\nu(Q))\\
&\leq& \sum_{k=1}^\infty \sum_{Q \in \mathcal{S}_k}\nu(Q)\big(\log\nu(Q) +k\log 2\big)\\
&=& \sum_{Q \in \mathcal{M}_r}\nu(Q)\log \nu(Q) +   \sum_{k=1}^\infty 3^d 2^{1-k}k \log 2.
\end{eqnarray*}
Since the right hand sum is finite, dividing by $\log r$ and taking the limit completes the argument for $q=1$.
\end{proof}

Let $\mu$ be  a positive finite Borel measure on $\Sigma^\infty$. Let $d_q^-$ be given by ~\eqref{def_adq-}.  Next, we give convenient alternative forms of $d_q^-$, which are closely related to the lower generalized $q$-dimensions.

Given $s>0$ and $r>0$, recall that
$$
\Sigma^*(s, r)=\{\mathbf{u}=u_1\ldots u_k \in \Sigma^*: \alpha_m(T_\mathbf{u})  \leq r < \alpha_m(T_{\bu^-}) \}.
$$

\begin{pro}\label{pro_pdq_eq}
For $q>1$,
\begin{eqnarray*}
  d_q^- &=& \sup \Big\{s: \limsup_{k\to \infty}\sum_{\bu\in\Sigma^k} \psi^s(T_\bu)^{1-q}\mu(\mathcal{C}_\bu)^q<\infty\Big\} \\
  &=& \sup \Big\{s: \sum_{k=1}^\infty \sum_{\bu\in\Sigma^k} \psi^s(T_\bu)^{1-q}\mu(\mathcal{C}_\bu)^q<\infty\Big\}.
\end{eqnarray*}

\end{pro}

\begin{proof}
We write
\begin{eqnarray*}
d_1 &=& \sup \Big\{s: \limsup_{k\to \infty}\sum_{\bu\in\Sigma^k} \psi^s(T_\bu)^{1-q}\mu(\mathcal{C}_\bu)^q<\infty\Big\}, \\
d_2 &=& \sup \Big\{s: \sum_{k=1}^\infty \sum_{\bu\in\Sigma^k} \psi^s(T_\bu)^{1-q}\mu(\mathcal{C}_\bu)^q<\infty\Big\}.
\end{eqnarray*}
It is sufficient to show that $d_q^-=d_1=d_2$.

Note that, for every $|\bu|=k$ and $0<s_1<s$, the following inequality holds
\begin{equation}\label{phis1}
\psi^{s_1}(T_\bu)\geq \alpha_m(T_\bu)^{(s_1-s)}\psi^s(T_\bu) \geq \alpha_+^{-k(s-s_1)}\psi^s(T_\bu).
\end{equation}
For each non-integral  $s<d_1$, we write that
$$
\limsup_{k\to \infty} \sum_{\bu\in\Sigma^k} \psi^s(T_\bu)^{1-q}\mu(\mathcal{C}_\bu)^q=M<\infty.
$$
It  immediately follows that for all $s_1<s$,
$$
\sum_{k=1}^\infty \sum_{\bu\in\Sigma^k}\psi^{s_1}(T_\bu)^{1-q}\mu(\mathcal{C}_\bu)^q \leq \sum_{k=0}^\infty M \alpha_+^{k(s-s_1)(q-1)}<\infty,
$$
and it implies that $d_1\leq d_2$.

For each $s<d_2$, we have that
$$
\sum_{k=1}^\infty \sum_{\bu\in\Sigma^k} \psi^s(T_\bu)^{1-q}\mu(\mathcal{C}_\bu)^q<\infty.
$$
Since
$$
\sum_{\bu\in\Sigma^*(s, r)}\psi^s(T_\bu)^{1-q}\mu(\mathcal{C}_\bu)^q\leq \sum_{k=1}^\infty \sum_{\bu\in\Sigma^k} \psi^s(T_\bu)^{1-q}\mu(\mathcal{C}_\bu)^q,
$$
we obtain that $d_2\leq d_q^-$.

Finally, for all non-integral $s<d_q^-$, suppose that
$$
\limsup_{r\to 0}\sum_{\bu\in\Sigma^*(s, r)}\psi^s(T_\bu)^{1-q}\mu(\mathcal{C}_\bu)^q=M.
$$
For each $s_1<s$, by ~\eqref{phis1}, we have  that for all  $\bu\in \Sigma^*(s,r)$,
$$
\psi^{s_1}(T_\bu) \geq r^{s_1-s} \psi^s(T_\bu).
$$
It follows that
$$
\sum_{\bu\in\Sigma^*(s, r)}\psi^{s_1}(T_\bu)^{1-q}\mu(\mathcal{C}_\bu)^q\leq Mr^t.
$$
where $t=(s-s_1)(q-1)>0$.  By choosing $\rho$ such that $\alpha_+<\rho<1$, we have that
$$
\Sigma^*=\bigcup_{k=0}^\infty \Sigma^k=\bigcup_{l=0}^\infty\Sigma^*(s,\rho^l),
$$
and  it implies that
\begin{eqnarray*}
 \sum_{\bu\in\Sigma^k}\psi^{s_1}(T_\bu)^{1-q}\mu(\mathcal{C}_\bu)^q
 &\leq& \sum_{l=0}^\infty \sum_{\bu\in\Sigma^*(s,\rho^l)}\psi^{s_1} (T_\bu)^{1-q}\mu(\mathcal{C}_\bu)^q   \\
 &\leq& \sum_{l=0}^\infty M\rho^{lt}  \\
 &<&\infty,
\end{eqnarray*}
for all $k>0$. It follows that
$$
\limsup_{k\to \infty} \sum_{\bu\in\Sigma^k} \psi^{s_1}(T_\bu)^{1-q}\mu(\mathcal{C}_\bu)^q<\infty,
$$
and we obtain that $d_q^-\leq d_1$. Therefore the conclusion holds.

\end{proof}

Next, we prove that $d_q^-$ and $d_q^+$ given by ~\eqref{def_adq-} and ~\eqref{def_adq+} are the upper bounds of lower and upper generalized $q$- dimension of measures supported on nonautonomous affine sets. The key argument is similar to  Theorem~\ref{thm_DqU}, but it is more complicated because of the complex geometric structure and the use of singular value functions.

\begin{proof}[Proof of Theorem~\ref{Dq_ub}]
We only give the proof for $\overline{D}_q(\mu^\w)\leq\min\{d_q^+,d\}$ since the proof for $\underline{D}_q(\mu^\w)\leq \min\{d_q^-,d\}$ is similar.

Let $B$ be a sufficiently large ball such that $J\subset B$. Given $\delta > 0$, there exists an integer $k>0$ such that $|\Psi_\mathbf{u}(B)| < \delta$ for all $|\mathbf{u}| \geq k$. Let $A$ be a covering set of $\Sigma^\infty$ such that $|\mathbf{u}| \geq k$ for each $\mathbf{u}\in A$. Then $E \subset \bigcup_{\mathbf{u}\in A} \Psi_\mathbf{u}(B)$. Every ellipsoid $\Psi_\mathbf{u}(B)$ is contained in a rectangular parallelepiped of side lengths $2|B|\alpha_1, \ldots, 2|B|\alpha_d$ where the $\alpha_i$ are the singular values of $T_\mathbf{u}$.

Given  $s>0$ and $r>0$.  Recall that
$$
\Sigma^*(s, r)=\{\mathbf{u}=u_1\ldots u_k \in \Sigma^*: \alpha_m(T_\mathbf{u})  \leq r < \alpha_m(T_{\bu^-}) \}.
$$
Let $m$ be integer such that $m-1<s\leq m$. We divide such a parallelepiped into
\begin{equation}\label{number}
N \leq (4|B|)^d\alpha_1 \alpha_2 \ldots \alpha_{m-1} \alpha_m^{1-m}= C_1 \psi^s(T_\mathbf{u}) \alpha_m^{-s}
\end{equation}
cubes of side $\alpha_m \leq r$,  where $C_1= (4|B|)^d$ is a constant, and  we write these cubes as $Q_1,\ldots,Q_N$. Let $\mu_\bu$ denote the restriction
of $\mu$ to the cylinder $\mathcal{C}_\bu$ and let $\mu_\bu^\w$
be its image under $\Pi^\w$.

For  $q>1$.  By Jensen's inequality, it follows that
$$
\mu(\mathcal{C}_\bu)^q = \mu_\bu^\w(T_\bu(B))^q=\bigg(\sum_{k=1}^N \mu_\bu^\w(Q_k)\bigg)^q \leq N^{q-1} \sum_{k=1}^N \mu_\bu^\w(Q_k)^q,
$$
and  by ~\eqref{number},
$$
\Big (C_1 \psi^s(T_\mathbf{u}) \alpha_m^{-s}\Big)^{1-q} \mu(\mathcal{C}_\bu)^q \leq \sum_{k=1}^N \mu_\bu^\w(Q_k)^q.
$$
Since $r$ is comparable with $\alpha_m$,  we have that
$$
r^{s(1-q)}\sum_{Q\in \mathcal{M}_r} \mu_\bu^\w(Q)^q \geq C_2 r^{s(1-q)} \sum_{k=1}^N \mu_\bu^\w(Q_k)^q \geq C_3 \psi^s(T_\mathbf{u})^{1-q} \mu(\mathcal{C}_\bu)^q,
$$
where the constants $C_2$ and  $C_3$ do not depend on $\bu$ or $\w$. For each $Q\in \mathcal{M}_r$,  by the power inequality, we have that
$$
\mu^\w(Q)^q = \bigg(\sum_{\bu\in\Sigma^s(r)} \mu_\bu^\w(Q)\bigg)^q \geq \sum_{\bu\in\Sigma^*(s,r)} \mu_\bu^\w(Q)^q.
$$
Summing it over $Q$, we have that
\begin{equation}\label{rsmuw}
r^{s(1-q)}\sum_{Q\in \mathcal{M}_r} \mu^\w(Q)^q \geq C_4\sum_{\bu\in\Sigma^*(s,r)}\psi^s(T_\bu)^{1-q}\mu(\mathcal{C}_\bu)^q,
\end{equation}
where the constant $C_4$ does not depend on $\bu$ or $\w$.

For all $s>d_q^+$, since $q>1$, by \eqref{def_adq+}, we have that
$$
 \liminf_{r\to 0}\sum_{\bu\in\Sigma^*(s, r)}\psi^s(T_\bu)^{1-q}\mu(\mathcal{C}_\bu)^q=\infty.
$$
Combining with ~\eqref{rsmuw}, this implies that
$$
s \geq \limsup_{r\to 0}\frac{\log \sum_{Q\in\mathcal{M}_r} \mu^\w(Q)^q}{(q-1)\log r}=\overline{D}_q(\mu^\w),
$$
and it immediately follows that
$$
\overline{D}_q(\mu^\w)\leq d_q^+.
$$

For  $0<q<1$. Both Jensen's inequality and the power inequality work in the reverse direction, and we obtain that
\begin{equation} \label{eq_rmpm}
r^{s(1-q)}\sum_{Q\in \mathcal{M}_r} \mu^\w(Q)^q \leq C\sum_{\bu\in\Sigma^*(s,r)}\psi^s(T_\bu)^{1-q}\mu(\mathcal{C}_\bu)^q,
\end{equation}
where the constant $C>0$ does not depend on $\bu$ or $\omega$.

For all $s>d_q^+$, by \eqref{def_adq+}, we have that

$$
\limsup_{r\to 0}\sum_{\bu\in\Sigma^*(s, r)}\psi^s(T_\bu)^{1-q}\mu(\mathcal{C}_\bu)^q<\infty.
$$
Combining with ~\eqref{eq_rmpm}, this implies that
$$
s \geq \limsup_{r\to 0}\frac{\log \sum_{Q\in\mathcal{M}_r} \mu^\w(Q)^q}{(q-1)\log r}=\overline{D}_q(\mu^\w),
$$
and it immediately follows that
$$
\overline{D}_q(\mu^\w)\leq d_q^+.
$$

For  $q=1$. For each $\bu\in\Sigma^*(s, r) $, the  ellipsoid $\Psi_\mathbf{u}(B)$ is divided into  cubes $Q_1,\ldots,Q_N$. By Jensen's inequality, we have that
\begin{eqnarray*}
\mu(\mathcal{C}_\bu)\log \mu(\mathcal{C}_\bu) &=&\Big(\sum_{k=1}^N \mu_\bu^\w(Q_k)\Big)\log \Big(\sum_{k=1}^N \mu_\bu^\w(Q_k)\Big)  \\ &\leq& \sum_{k=1}^N \mu_\bu^\w(Q_k) \log \mu_\bu^\w(Q_k) + \mu(\mathcal{C}_\bu)\log N.
\end{eqnarray*}
By \eqref{number}, it follows that
$$
\sum_{k=1}^N \mu_\bu^\w(Q_k) \log \mu_\bu^\w(Q_k) \geq \mu(\mathcal{C}_\bu) \log (N^{-1} \mu(\mathcal{C}_\bu)) \geq \mu(\mathcal{C}_\bu) \log \big[C_1r^s\psi^s(T_\bu)^{-1} \mu(\mathcal{C}_\bu)\big].
$$
Since $\sum_{\bu\in\Sigma^*(s,r)} \mu_\bu^\w(Q)=\mu^\w(Q)$, this implies  that
\begin{eqnarray}
\sum_{Q\in \mathcal{M}_r} \mu^\w(Q)\log \mu^\w(Q) &=& \sum_{Q\in \mathcal{M}_r} \Big(\sum_{\bu\in\Sigma^*(s,r)} \mu_\bu^\w(Q)\Big) \log \Big(\sum_{\bu\in\Sigma^*(s,r)} \mu_\bu^\w(Q)\Big) \nonumber \\
&\geq& \sum_{\bu\in\Sigma^*(s,r)} \sum_{Q\in \mathcal{M}_r} \mu_\bu^\w(Q) \log \mu_\bu^\w(Q) \nonumber \\
&\geq& \sum_{\bu\in\Sigma^*(s,r)} \mu(\mathcal{C}_\bu) \log \Big[C_1r^s\psi^s(T_\bu)^{-1}\mu(\mathcal{C}_\bu)\Big] \nonumber \\
&=& s\log r+ \sum_{\bu\in\Sigma^*(s,r)} \mu(\mathcal{C}_\bu) \log \Big[C_1\psi^s(T_\bu)^{-1}\mu(\mathcal{C}_\bu)\Big]. \label{ineq_summuQ}
\end{eqnarray}

For all $s>d_1^+$, by \eqref{def_adq+}, we have that
$$
 \liminf_{r\to 0}\sum_{\bu\in\Sigma^*(s, r)}\mu(\mathcal{C}_\bu)\log\Big(\psi^s(T_\bu)^{-1}\mu(\mathcal{C}_\bu) \Big)  = \infty.
$$
Combining with ~\eqref{ineq_summuQ}, this implies that
$$
s \geq \limsup_{r\to 0}\frac{\sum_{Q\in \mathcal{M}_r} \mu^\w(Q)\log \mu^\w(Q) }{\log r}=\overline{D}_1(\mu^\w),
$$
and it immediately follows that
$$
\overline{D}_q(\mu^\w)\leq d_q^+.
$$

Since $\overline{D}_q(\mu^\w)\leq d_q^+$  holds for all $q>0$, we  obtain that
$$
\overline{D}_q(\mu^\w)\leq\min\{d_q^+,d\},
$$
and the proof is completed.
\end{proof}

\section{Lower bounds for generalized $q$-dimensions of measures on nonautonomous affine sets}\label{sec 5}
\setcounter{equation}{0}
\setcounter{pro}{0}

Recall that $\mathcal{D}$ is a bounded region in $\R^d$. All translations $\omega_\bu\in  \mathcal{D} $ are  independent identically distributed random vectors
according to the probability measure $P=P_\bu$ which is absolutely continuous with respect to $d$-dimensional Lebesgue measure. The product probability measure $\mathbf{P} $ on
the family  $\omega =\{\omega_\bu : \bu \in  \Sigma^* \}$ is given by
$$
\mathbf{P} = \prod_{\bu\in \Sigma^*} P_\bu.
$$

Let $\mathbb{E}$ denote expectation. Given $\bjj \subseteq \Sigma^*$ we write $\mathcal{F} = \sigma\{\w_\bu : \bu \in \bjj\}$ for the sigma-field
generated by the random displacements $\w_\bu$ and write $\mathbb{E}(X|\mathcal{F})$ for the expectation of a random variable $X$ conditional on $\mathcal{F}$.
The following conclusion is proved in \cite{GM}, and the version of almost self-affine set is proved by Falconer in \cite{Falco88}.
\begin{lem}\label{lemE}
Let $s$ be non-integral such that $0 < s < d$. Then there exists a constant $C>0$ such that
$$
\mathbb{E}\Big(|\Pi^\w(\bu)-\Pi^\w(\bv)|^{-s}\  \Big| \ \mathcal{F}\Big) \leq  \frac{ C }{\psi^{s}(T_{\bu \wedge \bv})},
$$
for all $\bu,\bv\in \Sigma^\infty$, $\bu \neq \bv$, where $\mathcal{F} =\sigma\{\w_\bu : \bu\in \Lambda\}$ for any subset $\Lambda $ of $\Sigma^*$ such that $\bv|k+1, \bv|k+2,\ldots\in \Lambda$ and $\bu|k+2, \bu|k+3,\ldots \in \Lambda $ but $\bu|k+1\notin \Lambda$, where $|\bu\wedge \bv|=k$.
\end{lem}

Given $\bu_1, \ldots, \bu_n \in  \Sigma^\infty$, we denote the join set of $\bu_1, \ldots, \bu_n$ by
\begin{equation}
\bigwedge(\bu_1, \ldots, \bu_n) = \{\bu_i \wedge \bu_j : i \neq j\},
\end{equation}
which is the collection of join points with repetitions counted by multiplicity in a natural way. We define multienergy kernels by forming products of the singular value functions at the vertices of join sets. For $\bu_1, \ldots, \bu_n \in  \Sigma^\infty$, let
\begin{equation}
\psi^s(\bu_1, \ldots, \bu_n)= \psi^s(T_{\bw_1}) \psi^s(T_{\bw_2}) \ldots \psi^s(T_{\bw_{n-1}}).
\end{equation}
where $\{\bw_1,\ldots,\bw_{n-1}\}=\bigwedge(\bu_1, \ldots, \bu_n)$. The multienergy kernel $\psi^s(\bu_1, \ldots, \bu_n)$ is the key to study the generalized $q$-dimension of nonautonomous affine measures.

\begin{pro}\label{calE}
Given an integer $n\geq 1$ and a real $q$ such that $1<q\leq n+1$. For all non-integral $s$ satisfying $0<s<d$, there exist two constants $C>0$ and $r_0>0$ such that for all $\bv\in\Sigma^\infty$ and $0<r\leq r_0$,
\begin{eqnarray*}
  &&\hspace{-1cm}\mathbb{E}\Big(\int \mu^\w(B(x,r))^{q-1} d\mu^\w(x)\Big) \\
  &\leq& Cr^{s(q-1)}\int\left[\int \ldots \int \psi^s(\bu_1,\ldots,\bu_n,\bv)^{-1} d\mu(\bu_1)\ldots d\mu(\bu_n)\right]^{(q-1)/n}d\mu(\bv).
\end{eqnarray*}
\end{pro}

\begin{proof}
We may renumber the points $\bu_1,\ldots,\bu_n$ in such a manner that $\{\bu_1\wedge\bu_2, \ldots, \bu_{n-1}\wedge\bu_n, \bu_n\wedge\bv\}$are precisely the points of the join set $\bigwedge(\bu_1,\ldots,\bu_n,\bv)$, including repeat points. This renumbering does not affect the value of $\psi^s(\bu_1,\ldots,\bu_n,\bv)$.

Since $n/(q-1)\geq 1$, by  Fubini's theorem and Jensen's inequality, we have that
\begin{eqnarray}\label{E_intel}
\mathbb{E} \Big(\int \mu^\w(B(x,r))^{q-1}d\mu^\w(x)\Big) &=& \mathbb{E}\Big(\int \mu^\w(B(\Pi^\w(\bv),r))^{q-1} d\mu(\bv)\Big)  \nonumber \\
&\leq&\int \Big[\mathbb{E} \Big(\mu^\w(B(\Pi^\w(\bv),r))^n\Big)\Big]^{(q-1)/n} d\mu(\bv).
\end{eqnarray}
The key is to estimate the expectation $\mathbb{E} (\mu^\w(B(\Pi^\w(\bv),r))^n)$. Again by Fubini's theorem, we have that
\begin{eqnarray}\label{leqP}
&&\hspace{1cm}\mathbb{E}\Big(\mu^\w(B(\Pi^\w(\bv),r))^n\Big)
= \mathbb{E}\Big(\mu\{\bu:|\Pi^\w(\bu)-\Pi^\w(\bv)|\leq r\}^n\Big) \\
&&=\!\! (\mathbf{P}\times \mu \times \ldots \times \mu)\{(\w,\bu_1,\ldots,\bu_n): |\Pi^\w(\bu_1)-\Pi^\w(\bv)|\leq r, \ldots,|\Pi^\w(\bu_n)-\Pi^\w(\bv)|\leq r\} \nonumber\\
&&= \int\ldots\int \mathbf{P}\{\w:|\Pi^\w(\bu_1)-\Pi^\w(\bv)|\leq r,\ldots,|\Pi^\w(\bu_n)-\Pi^\w(\bv)|\leq r\} d\mu(\bu_1)\ldots d\mu(\bu_n) \nonumber\\
&&\leq \int\ldots\int \mathbf{P}\{|\Pi^\w(\bu_1)-\Pi^\w(\bu_2)|\leq 2r,\ldots, |\Pi^\w(\bu_n)-\Pi^\w(\bv)|\leq 2r\} d\mu(\bu_1)\ldots d\mu(\bu_n) \nonumber\\
&&\leq 2^n r^{sn} \int\ldots\int \mathbb{E}(|\Pi^\w(\bu_1)-\Pi^\w(\bu_2)|^{-s}\ldots |\Pi^\w(\bu_n)-\Pi^\w(\bv)|^{-s}) d\mu(\bu_1)\ldots d\mu(\bu_n).     \nonumber
\end{eqnarray}

We  define a sequence
of $\sigma$-algebras $\mathcal{F}_1 \supset \mathcal{F}_2 \supset \ldots \supset \mathcal{F}_n$ by
$$
\mathcal{F}_l=\sigma\{\w_\bu:\bu\neq \bu_1|{k_1+1},\ldots \bu_l|{k_l+1}\},
$$
where $k_l=|\bu_l\wedge \bu_{l+1}|$ for $l=1,2,\ldots, n-1$ and $k_n=|\bu_n\wedge \bv|$. We write
$$
X_l^\w=|\Pi^\w(\bu_l)-\Pi^\w(\bu_{l+1})|^{-s},
$$
 for $l=1,\ldots,n-1$ and
$$
X_n^\w=|\Pi^\w(\bu_n)-\Pi^\w(\bv)|^{-s},
$$
and it is clear that  $X_{l+1}^\w,\ldots,X_n^\w$ are all $\mathcal{F}_l$-measurable.  Applying  the tower property of conditional expectation and Lemma~\ref{lemE} $n$ times, we obtain that
\begin{eqnarray*}
\mathbb{E}(X_1^\w \ldots X_n^\w|\mathcal{F}_n)&=& \mathbb{E}(\mathbb{E}(X_1^\w \ldots X_n^\w|\mathcal{F}_1)|\mathcal{F}_n)\\
&=& \mathbb{E}(\mathbb{E}(X_1^\w|\mathcal{F}_1)X_2^\w \ldots X_n^\w|\mathcal{F}_n)\\
&\leq&\mathbb{E}(C_1\psi^s(T_{\bu_1\wedge\bu_2})^{-1} X_2^\w \ldots X_n^\w|\mathcal{F}_n)\\
&=&C_1\psi^s(T_{\bu_1\wedge\bu_2})^{-1} \mathbb{E}(X_2^\w \ldots X_n^\w|\mathcal{F}_n)\\
&\leq& C_1^n \psi^s(T_{\bu_1\wedge\bu_2})^{-1}\ldots \psi^s(T_{\bu_{n-1}\wedge\bu_n})^{-1} \psi^s(T_{\bu_n\wedge\bv})^{-1} \\
&=& C_1^n \psi^s(\bu_1,\ldots,\bu_n,\bv)^{-1}.
\end{eqnarray*}

Again by the tower property of conditional expectation, it follows that
\begin{eqnarray*}
\mathbb{E}(|\Pi^\w(\bu_1)-\Pi^\w(\bu_2)|^{-s}\ldots |\Pi^\w(\bu_n)-\Pi^\w(\bv)|^{-s})
& =& \mathbb{E}(\mathbb{E}(X_1^\w \ldots X_n^\w|\mathcal{F}_n)|\mathcal{F}_n) \\
& \leq& C_1^n \psi^s(\bu_1,\ldots,\bu_n,\bv)^{-1}.
\end{eqnarray*}
Combining this with ~\eqref{leqP}, we obtain that
\begin{equation}\label{lemEphi}
  \mathbb{E}\big(\mu^\w(B(\Pi^\w(\bv),r)^n\big)\leq C_2r^{sn}\int \ldots \int \psi^s(\bu_1,\ldots,\bu_n,\bv)^{-1} d\mu(\bu_1)\ldots d\mu(\bu_n),
\end{equation}
where $C_2$ is a constant.
Combining \eqref{E_intel} and ~\eqref{lemEphi} together,  the conclusion holds.
\end{proof}

\iffalse
we have that
\begin{eqnarray*}
&&\hspace{-1.5cm}\mathbb{E} \Big(\int \mu^\w(B(x,r))^{q-1}d\mu^\w(x)\Big) = \mathbb{E}\Big(\int \mu^\w(B(\Pi^\w(\bv),r))^{q-1} d\mu(\bv)\Big)\\
&\leq&\int \Big[\mathbb{E} \Big(\mu^\w(B(\Pi^\w(\bv),r))^n\Big)\Big]^{(q-1)/n} d\mu(\bv)\\
&\leq& Cr^{s(q-1)}\int\left[\int \ldots \int \psi^s(\bu_1,\ldots,\bu_n,\bv)^{-1} d\mu(\bu_1)\ldots d\mu(\bu_n)\right]^{(q-1)/n}d\mu(\bv).
\end{eqnarray*}
\end{proof}
\fi

We identify the $\Sigma^*$ with the vertices of a rooted tree with root $\emptyset$, in the obvious way. The edges of this tree join each vertex $\bu\in \Sigma^k$ to its $n_{k+1}$ `children' $\bu 1, \ldots, \bu n_{k+1}$. It is natural to think of the join points in $\bigwedge(\bu_1, \ldots, \bu_n)$ as vertices of the tree $\Sigma^*$ where the paths from $\emptyset$ to the $\bu_i \wedge \bu_j$ meet.

A join set $\bigwedge(\bu_1, \ldots, \bu_n)$ with root $\bw\in\Sigma^*$ consists of a family of vertices $\{\bv_1, \ldots, \bv_r\}$ of $\Sigma^*$, with repetitions allowed, such that $\bv_j \succeq \bw$ for all $j$ and with the property that $\bv_i \wedge \bv_j \in \bigwedge(\bv_1, \ldots, \bv_n)$ for all $\bv_i,\bv_j \in \bigwedge(\bu_1, \ldots, \bu_n)$. The root may or may not be a vertex of the join set. The number $r+1$ is called the {\it spread} of the join set. The multiplicity of a given vertex is the number of times it occurs in $\bigwedge(\bu_1, \ldots, \bu_n)$.

Given $\bu_1, \ldots, \bu_n \in  \Sigma^\infty$, let $\bjj=\bigwedge(\bu_1, \ldots, \bu_n)$ be a join set of $\bu_1, \ldots, \bu_n$ with root  $\bw\in\Sigma^*$. A join class $\mathcal{J}$  with root $\bw \in \Sigma^*$ is an equivalence class of join sets all with root $\bw$, two such join sets $\bjj$ and $\bjj'$ being equivalent if there is an automorphism of the rooted subtree of $\Sigma^*$ with root $\bw$ that maps $\bjj$ onto $\bjj'$(preserving multiplicities). The spread of a join class $\mathcal{J}$ is the common value of the spreads of all $\bjj \in \mathcal{J}$ .

Recall that the level of a vertex $\bw \in \Sigma^*$ is just $|\bw|$. Given a joint set $\bjj = \{\bv_1,\ldots,\bv_r\}$, we define the set of levels $L(\bjj)$ of $\bjj$ by $\{|\bv_1|, \ldots, |\bv_r|\}$, allowing repetitions, and we define the set of levels $L(\mathcal{J})$ of the join class $\mathcal{J}$ is the common set of levels of the join sets in the class.

Note that in the rest of the paper, the product is always over the set of levels in a join class. The symbol $[n-1]$ above the product sign means that there are $n-1$ terms in this product.

\begin{pro}
Given a real $q>1$. Let $n$ be the integer such that $q-1 \leq n < q$. Let $\mathcal{J}$ be a join class with root $\bw$ and spread $n$. Then
\begin{eqnarray}
&& \hspace{-2cm}\int_{\bigwedge(\bu_1,\ldots,\bu_n)\in\mathcal{J}} \psi^s(\bu_1,\ldots,\bu_n)^{-1} d\mu(\bu_1)\ldots d\mu(\bu_n) \nonumber\\
&\leq& \mu(\mathcal{C}_\bw)^{(q-n)/(q-1)} \prod_{l\in L(\mathcal{J})}^{[n-1]}\Big(\sum_{\substack{|\bw'|=l \\ \bw'\succeq \bw}} \psi^s(T_{\bw'})^{1-q} \mu(\mathcal{C}_{\bw'})^q\Big)^{1/(q-1)}. \label{eq_int1}
\end{eqnarray}
\end{pro}
\begin{proof}
We write
$$
I_n=\int_{\bigwedge(\bu_1,\ldots,\bu_n)\in\mathcal{J}} \psi^s(\bu_1,\ldots,\bu_n)^{-1} d\mu(\bu_1)\ldots d\mu(\bu_n)
$$
For  $n=1$, it implies that $1<q \leq 2$, and   we have that
$$
I_1= \int_{\bu_1 \succeq \bw} \psi^s(\bu_1) d\mu(\bu_1) = \int_{\bu_1 \succeq \bw} d\mu(\bu_1) =  \mu(\mathcal{C}_{\bw}).
$$
Hence the inequality~\eqref{eq_int1} holds for $n=1$.

We proceed by induction on the number of distinct vertices of $\mathcal{J}$. To start the inductive
process, suppose that the join sets in $\mathcal{J}$ consist of a single vertex $\bw$ of multiplicity $n-1$ for some $n \geq 2$. First assume that $\bw$ is itself the root of the join sets of $\mathcal{J}$. Then
\begin{eqnarray}
I_n &=& \int \int_{\bigwedge(\bu_1,\ldots,\bu_n)\in\mathcal{J}} \psi^s(T_\bw)^{-(n-1)} d\mu(\bu_1)\ldots d\mu(\bu_n) \nonumber\\
&\leq& \psi^s(T_\bw)^{-(n-1)}\mu(\mathcal{C}_\bw)^n \nonumber\\
&=& \mu(\mathcal{C}_\bw)^{(q-n)(q-1)} \big(\psi^s(T_\bw)^{1-q}\mu(\mathcal{C}_\bw)^q\big)^{(n-1)/(q-1)}.
\end{eqnarray}
Since $L(\mathcal{J})$ has just one level $|\bw|$ which is of multiplicity $n-1$, the conclusion follows immediately.

Next, suppose that the join class $\mathcal{J}$ has root $\bh$ and contains join sets $\mathcal{J}$ consisting of a single vertex $\bw$, distinct from $\bh$, of multiplicity $n-1$ at level $l$.

\begin{eqnarray}
I_n &\leq& \sum_{\substack{|\bw|=l \\ \bw\succeq \bh}} \mu(\mathcal{C}_\bw)^{(q-n)(q-1)} (\psi^s(T_\bw)^{1-q}\mu(\mathcal{C}_\bw)^q)^{1/(q-1)} \nonumber\\
&\leq& \Big(\sum_{\substack{|\bw|=l \\ \bw\succeq \bh}} \mu(\mathcal{C}_\bw) \Big)^{(q-n)(q-1)} \Big(\sum_{\substack{|\bw|=l \\ \bw\succeq \bh}} \psi^s(T_\bw)^{1-q}\mu(\mathcal{C}_\bw)^q \Big)^{(n-1)/(q-1)} \nonumber\\
&=& \mu(\mathcal{C}_\bh)^{(q-n)(q-1)} \Big(\sum_{\substack{|\bw|=l \\ \bw\succeq \bh}} \psi^s(T_\bw)^{1-q}\mu(\mathcal{C}_\bw)^q \Big)^{(n-1)/(q-1)}.
\end{eqnarray}

We now assume inductively that ~\eqref{eq_int1} holds for all join sets with fewer than $k$ distinct vertices for some $k \geq 2$. Let $\mathcal{J}$ be a join class with root $\bw$ and $k \geq 2$ distinct vertices and spread $n$ where $n \leq q$. Again, first consider the case where the root $\bw$ belongs to the join sets in $\mathcal{J}$ as the `top' vertex. In each join set $\bjj \in  \mathcal{J}$ there is a (possibly empty) set of $r \geq 0$ vertices $\{\bw_1, \ldots, \bw_r\}$ in $\mathcal{J}$ distinct from and `immediately below' $\bj$, that is with the path joining $\bw_i$ to $\bw$ in the tree $I$ containing no other vertices of $\bjj$. For a given class J these sets of vertices (with
multiplicity) are equivalent under automorphisms of the tree that fix the root $\bw$.

For each $i$, the join set $\bjj \in \mathcal{J}$ induces a join set that we denote by $\bjj_i$ with root $\bj$ and vertices $\{\bh \in \mathcal{J} : \bh \succeq \bw_i\}$, that is the vertices of $J$ below and including $\bw_i$. These join sets are
equivalent under automorphisms of the tree that fix the root $\bw$, and we write $\mathcal{J}_i = \{\bjj_i : \bjj \in \mathcal{J}\}$ for this equivalence class of join sets, which has spread $n_i \geq 2$, say, and set of levels $L_i$(counted with repetitions).

Let $  n_1+\ldots+n_r+t=n,$ $t\geq 0$. To integrate over $\{\bu_1, \ldots, \bu_n\}$ such that $\bigwedge(\bu_1, \ldots , \bu_n) \in \mathcal{J}$, we decompose the integral so that for each $\bjj \in \mathcal{J}$, for each $i (1 \leq i \leq r)$, the $n_i$ integration variables, $\{\bu_1^i, \ldots, \bu_{n_i}^i\}$ say, are such that
$\bigwedge(\bu_1^i, \ldots, \bu_{n_i}^i)=\bjj_i$ and $t$ of them, $\{\bu_1^0, \ldots, \bu_l^0\}$, such that $\bu_l^0 \wedge \bu_k = \bw$ for all $\bu_l^0 \neq \bu_k$. Thus, noting that the multiplicity of $\bw$ is $r+t-1$ and applying the inductive assumption~\eqref{eq_int1} to join sets in $\mathcal{J}_i$ for each $i$,
\begin{eqnarray*}
&& \hspace{-1cm} I_n  \leq \psi^s(T_\bw)^{-(r+t-1)}\mu(\mathcal{C}_\bw)^t \int_{\bigwedge(\bu_1^1,\ldots,\bu_n^1)\in\mathcal{J}_1} \psi^s(\bu_1^1,\ldots,\bu_n^1)^{-1} d\mu(\bu_1^1)\ldots d\mu(\bu_{n_1}^1) \\
&& \hspace{3cm}\times \ldots \times \int_{\bigwedge(\bu_1^r,\ldots,\bu_n^r)\in\mathcal{J}_r} \psi^s(\bu_1^r,\ldots,\bu_n^r)^{-1} d\mu(\bu_1^r)\ldots d\mu(\bu_{n_1}^r) \\
&\leq& \psi^s(T_\bw)^{1-r-t}\mu(\mathcal{C}_\bw)^t\mu(\mathcal{C}_\bw)^{(q-n_1)(q-1)} \prod_{l\in L_1}^{[n_1-1]}\Big(\sum_{\substack{|\bw'|=l \\ \bw'\succeq \bw}} \psi^s(T_{\bw'})^{1-q}\mu(\mathcal{C}_{\bw'})^q\Big)^{1/(q-1)} \\
&& \hspace{3cm}\times \ldots \times \mu(\mathcal{C}_\bw)^{(q-n_r)(q-1)} \prod_{l\in L_r}^{[n_r-1]}\Big(\sum_{\substack{|\bw'|=l \\ \bw'\succeq \bw}} \psi^s(T_{\bw'})^{1-q} \mu(\mathcal{C}_{\bw'})^q\Big)^{1/(q-1)}\\
&\leq& \mu(\mathcal{C}_\bw)^{\frac{q-n_1-\ldots-n_r-t}{q-1}} (\psi^s(T_\bw)^{(1-q)}\mu(\mathcal{C}_\bw))^{\frac{r+t-1}{q-1}}  \prod_{l\in L_1\cup\ldots\cup L_r}^{n_1+\ldots+n_r-r} \Big(\sum_{\substack{|\bw'|=l \\ \bw'\succeq \bw}} \psi^s(T_{\bw'})^{1-q}\mu(\mathcal{C}_{\bw'})^q\Big)^{1/(q-1)} \\
&\leq& \mu(\mathcal{C}_\bw)^{(q-n)(q-1)} \prod_{l\in L(\mathcal{J})}^{[n-1]}\Big(\sum_{\substack{|\bw'|=l \\ \bw'\succeq \bw}} \psi^s(T_{\bw'})^{1-q}\mu(\mathcal{C}_{\bw'})^q\Big)^{1/(q-1)}.
\end{eqnarray*}
Hence the inequality~\eqref{eq_int1} holds for  the root $\bw$ is a vertex of the join sets in $\mathcal{J}$.

Finally, if the root $\bh$ is not a vertex of the join sets in $\mathcal{J}$, then summing ~\eqref{eq_int1} over join sets with top vertex $\bw \succeq \bh$ with $|\bw| = l'$ and using H\"{o}lder's inequality,
\begin{eqnarray*}
I_n  &\leq& \sum_{\substack{|\bw|=l'\\\bw \succeq \bh}} \mu(\mathcal{C}_\bw)^{(q-n)(q-1)} \prod_{l\in L(\mathcal{J})}^{[n-1]}\Big(\sum_{\substack{|\bw'|=l \\ \bw'\succeq \bw}} \psi^s(T_{\bw'})^{1-q}\mu(\mathcal{C}_{\bw'})^q\Big)^{1/(q-1)} \\
&\leq& \Big(\sum_{\substack{|\bw|=l'\\ \bw \succeq \bh}} \mu(\mathcal{C}_\bw)\Big)^{(q-n)(q-1)} \prod_{l\in L(\mathcal{J})}^{[n-1]}\Big(\sum_{\substack{|\bw|=l'\\ \bw \succeq \bh}} \sum_{\substack{|\bw'|=l \\ \bw'\succeq \bw}} \psi^s(T_{\bw'})^{1-q}\mu(\mathcal{C}_{\bw'})^q\Big)^{1/(q-1)}\\
&=& \mu(\mathcal{C}_\bh)^{(q-n)(q-1)} \prod_{l\in L(\mathcal{J})}^{[n-1]}\Big(\sum_{\substack{|\bw'|=l \\ \bw'\succeq \bh}} \psi^s(T_{\bw'})^{1-q}\mu(\mathcal{C}_{\bw'})^q\Big)^{1/(q-1)},
\end{eqnarray*}
and we complete the induction and the proof.
\end{proof}

Let $0 \leq k_1 < k_2 < \ldots < k_p$ be levels, where $1 \leq p \leq n$. For $\bv \in \Sigma^\infty$,  write $\bv_r = \bv|k_r$. For each $r = 1, \ldots, p$, let $\mathcal{J}_r$ be a given join class with root at level $k_r$ and spread $m_r$.
Write $\mathcal{J}_r(\bv_r)$ for the corresponding join class with the root mapped to $\bv_r$, that is so that the tree automorphisms of $\Sigma^*$ that map $\bv_r$ to $\bv_r'$ map the join sets in $\mathcal{J}_r(\bv_r)$ onto those in $\mathcal{J}_r'(\bv_r)$ in a bijective manner.

We need to include the cases where $\mathcal{J}_r(\bv_r)$ is a join class with root $\mathbf{v}_r$ and of spread 1. In this case we interpret integration over $\bigwedge(\bu_l) \in \mathcal{J}_r(\bv_r)$ as integration over all $\bu_l$ such that
$\bu_l \wedge \bv = \bv_r$, and we take $\psi^s(\bu_l) = 1$ where this occurs in the next proof.

\begin{pro}\label{pro_intest}
Let $q>1$ and let $n$ be an integer with $q-1\leq n <q$. With notation as above,
for each $r = 1, \ldots, p$ let $\mathcal{J}_r$ be a join class with root at level $k_r$ and spread $m_r \geq 1$, with $m_1 + \ldots + m_p = n$. Then
\begin{eqnarray}
&&\hspace{-2cm}\int_{\Sigma^\infty} \bigg[\int_{\bigwedge(\bu_1,\ldots,\bu_{m_1})\in \mathcal{J}_1(\bv_1)}\ldots \nonumber\\
&&\int_{\bigwedge(\bu_{n-m_p+1},\ldots,\bu_n)\in \mathcal{J}_p(\bv_p)} \psi^s(\bu_1,\ldots,\bu_n,\bv)^{-1} d\mu(\bu_1)\ldots d\mu(\bu_n)\bigg]^{(q-1)/n} d\mu(\bv) \nonumber\\
&\leq& \prod_{l\in L}^{[n]}\Big(\sum_{|\bw|=l} \psi^s(T_\bw)^{1-q}\mu(\mathcal{C}_\bw)^q \Big)^{1/n},
\end{eqnarray}
where $L$ denotes the aggregate set of levels of $\{L(\mathcal{J}_1),\ldots,L(\mathcal{J}_p), k_1,\ldots,k_p\}$.
\end{pro}

\begin{proof}
For each $r = 1, \ldots, p$,  we write
\begin{eqnarray*}
&&\hspace{-0.8cm} I_r = \int_{\bv\succeq \bv_r}\bigg[\int_{\bigwedge(\bu_1^r,\ldots,\bu_{m_r}^r) \in \mathcal{J}_r(\bv_r)} \ldots \int_{\bigwedge(\bu_1^p,\ldots,\bu_{m_p}^p) \in \mathcal{J}_p(\bv_p)} \psi^s(\bu_1^r,\ldots,\bu_{m_r}^r)^{-1} \psi^s(T_{\bv_r})^{-1} \times \ldots \\
&& \times \psi^s(\bu_1^p,\ldots,\bu_{m_p}^p)^{-1} \psi^s(T_{\bv_p})^{-1} d\mu(\bu_1^r) \ldots d\mu(\bu_{m_r}^r) \ldots d\mu(\bu_1^p) \ldots d\mu(\bu_{m_p}^p) \bigg]^{(q-1)/n}d\mu(\bv).
\end{eqnarray*}
We prove it by induction on $r$, starting with $r = p$ and working up to $r = 1$, taking as the inductive hypothesis:
For all $\bw_r \in \Sigma^{k_r}$,
\begin{equation}\label{hyp}
I_r \leq \mu(\mathcal{C}_{\bv_r})^{(n-n_r)/n} \prod_{l\in L_r}^{[n_r]}\Big(\sum_{|\bw|=l,\bw \succeq \bv_r} \psi^s(T_\bw)^{1-q}\mu(\mathcal{C}_\bw)^q \Big)^{1/n},
\end{equation}
where $n_r=m_r+\ldots+m_p$ and $L_r$ denotes the set of $\{L(\mathcal{J}_r),\ldots,L(\mathcal{J}_p), k_r,\ldots,k_p\}$ counted by multiplicity(so that $L_r$ consists of $m_r+\ldots+m_p=n_r$ levels).
To start the induction, we apply~\eqref{eq_int1} to get
\begin{eqnarray}
&&\hspace{-0.8cm} I_p \leq \int_{\bv \succeq \bv_p} \hspace{-0.2cm}\bigg[ \mu(\mathcal{C}_{\bv_p})^{(q-m_p)/(q-1)}\psi^s(T_{\bv_p})^{-1} \hspace{-0.3cm}\prod_{l\in L(\mathcal{J}_p(\bv_p))}^{[m_p-1]} \hspace{-0.2cm}\Big(\sum_{\substack{|\bw|=l\\ \bw\succeq \bv_p}} \psi^s(T_\bw)^{1-q}\mu(\mathcal{C}_\bw)^q\Big)^{1/(q-1)} \bigg]^{(q-1)/n} \hspace{-0.3cm}d\mu(\bv) \nonumber\\
&&= \mu(\mathcal{C}_{\bv_p})^{(n-m_p)/n} (\psi^s(T_{\bv_p})^{1-q}\mu(\mathcal{C}_{\bv_p})^q)^{1/n} \prod_{l\in L(\mathcal{J}_p(\bv_p))}^{[m_p-1]} \Big(\sum_{|\bw|=l,\bw\succeq \bv_p} \psi^s(T_\bw)^{1-q}\mu(\mathcal{C}_\bw)^q\Big)^{1/n} \nonumber\\
&&= \mu(\mathcal{C}_{\bv_p})^{(n-n_p)/n} \prod_{l\in L_p}^{[n_p]} \Big(\sum_{|\bw|=l,\bw\succeq \bv_p} \psi^s(T_\bw)^{1-q}\mu(\mathcal{C}_\bw)^q\Big)^{1/n}. \nonumber
\end{eqnarray}
This establishes the inductive hypothesis ~\eqref{hyp} when $r = p$.

Now assume that ~\eqref{hyp} is valid for $r = k, \ldots, p$ for some $2 \leq k \leq p$. Then
\begin{eqnarray}
&&\hspace{-0.8cm} I_{k-1} = \psi^s(T_{\bv_{k-1}})^{(1-q)/n} \nonumber\\
&& \times \bigg[\int_{\bigwedge(\bu_1^{k-1}, \ldots, \bu_{m_{k-1}}^{k-1})\in\mathcal{J}_{k-1}(\bv_{k-1})} \psi^s(\bu_1^{k-1},\ldots,\bu_{m_{k-1}}^{k-1})^{-1} d\mu(\bu_1^{k-1}) \ldots d\mu(\bu_{m_{k-1}}^{k-1}) \bigg]^{(q-1)/n} \nonumber\\
&& \times \sum_{\bv_k \succeq \bv_{k-1}} \int_{\bv\succeq\bv_k} \bigg[\int_{\bigwedge(\bu_1^k, \ldots, \bu_{m_k}^k)\in\mathcal{J}_k(\bv_k)} \ldots \int_{\bigwedge(\bu_1^p, \ldots, \bu_{m_p}^p)\in\mathcal{J}_p(\bv_p)} \nonumber\\
&& \hspace{3.5cm} \psi^s(\bu_1^k,\ldots,\bu_{m_k}^k)^{-1} \psi^s(T_{\bv_k})^{-1} \ldots \psi^s(\bu_1^p,\ldots,\bu_{m_p}^p)^{-1} \psi^s(T_{\bv_p})^{-1}\nonumber\\
&& \hspace{3.5cm}d\mu(\bu_1^k) \ldots d\mu(\bu_{m_k}^k) \ldots d\mu(\bu_1^p) \ldots d\mu(\bu_{m_p}^p) \bigg]^{(q-1)/n}d\mu(\bv) \nonumber\\
&\leq& \psi^s(T_{\bv_{k-1}})^{(1-q)/n} \mu(\mathcal{C}_{\bv_{k-1}})^{(q-m_{k-1})/n} \prod_{l\in L(\mathcal{J}_{k-1}(\bv_{k-1}))}^{[m_{k-1}-1]} \Big(\sum_{|\bw|=l,\bw\succeq \bv_{k-1}} \psi^s(T_\bw)^{1-q}\mu(\mathcal{C}_\bw)^q\Big)^{1/n} \nonumber\\
&& \times \sum_{\bv_k \succeq \bv_{k-1}} \mu(\mathcal{C}_{\bv_k})^{(n-n_k)/n} \prod_{l\in L_k}^{[n_k]} \Big(\sum_{|\bw|=l,\bw\succeq \bv_k} \psi^s(T_\bw)^{1-q}\mu(\mathcal{C}_\bw)^q\Big)^{1/n}. \nonumber
\end{eqnarray}

Using H\"{o}lder's inequality for each $\bv_{k-1}$:
\begin{eqnarray*}
&&\sum_{\bv_k \succeq \bv_{k-1}} \mu(\mathcal{C}_{\bv_k})^{(n-n_k)/n} \prod_{l\in L_k}^{[n_k]} \Big(\sum_{|\bw|=l,\bw\succeq \bv_k} \psi^s(T_\bw)^{1-q}\mu(\mathcal{C}_\bw)^q\Big)^{1/n} \\
&\leq& \Big(\sum_{\bv_k \succeq \bv_{k-1}} \mu(\mathcal{C}_{\bw_k})\Big)^{(n-n_k)/n} \prod_{l\in L_k}^{[n_k]} \Big(\sum_{\bv_k \succeq \bv_{k-1}} \sum_{|\bw|=l,\bw\succeq \bv_k} \psi^s(T_\bw)^{1-q}\mu(\mathcal{C}_\bw)^q\Big)^{1/n} \\
&=& \mu(\mathcal{C}_{\bv_{k-1}})^{(n-n_k)/n} \prod_{l\in L_k}^{[n_k]} \Big(\sum_{|\bw|=l,\bw\succeq \bv_{k-1}} \psi^s(T_\bw)^{1-q}\mu(\mathcal{C}_\bw)^q\Big)^{1/n}.
\end{eqnarray*}
Hence
\begin{eqnarray*}
&& \hspace{-1cm}I_{k-1} \leq \mu(\mathcal{C}_{\bv_{k-1}})^{(n-n_k-m_{k-1})/n} \big(\psi^s(T_{\bv_{k-1}})^{1-q}\mu(\mathcal{C}_{\bv_{k-1}})^q \big)^{1/n} \\
&& \times \prod_{l\in L(\mathcal{J}_{k-1}(\bv_{k-1}))}^{[m_{k-1}-1]} \Big(\sum_{\substack{|\bw|=l \\ \bw\succeq \bv_{k-1}}} \psi^s(T_\bw)^{1-q}\mu(\mathcal{C}_\bw)^q\Big)^{1/n} \prod_{l\in L_k}^{[n_k]} \Big(\sum_{\substack{|\bw|=l \\ \bw\succeq \bv_{k-1}}} \psi^s(T_\bw)^{1-q}\mu(\mathcal{C}_\bw)^q\Big)^{1/n} \\
&=& \mu(\mathcal{C}_{\bv_{k-1}})^{(n-n_k-m_{k-1})/n} \prod_{l\in L_{k-1}}^{[m_{k-1}+n_k]} \Big(\sum_{\substack{|\bw|=l \\ \bw\succeq \bv_{k-1}}} \psi^s(T_\bw)^{1-q}\mu(\mathcal{C}_\bw)^q\Big)^{1/n} \\
&=& \mu(\mathcal{C}_{\bv_{k-1}})^{(n_{k-1})/n} \prod_{l\in L_{k-1}}^{[n_{k-1}]} \Big(\sum_{\substack{|\bw|=l \\ \bw\succeq \bv_{k-1}}} \psi^s(T_\bw)^{1-q}\mu(\mathcal{C}_\bw)^q\Big)^{1/n},
\end{eqnarray*}
which is ~\eqref{hyp} with $r=k-1$.

Note that $n_1=n$. Taking $r=1$, we have that
\begin{equation}
I_1 \leq \prod_{l\in L_1}^{[n]}\Big(\sum_{\substack{|\bw|=l \\ \bw\succeq\bv_1}} \psi^s(T_\bw)^{1-q}\mu(\mathcal{C}_\bw)^q \Big)^{1/n}.
\end{equation}
Summing over all $\bv_1$ at level $k_1$ and using H\"{o}lder's inequality again, we obtain the conclusion.
\end{proof}

Next, we put  all the estimates together, and obtain an almost sure lower bound for the lower $q$-dimension of measures on nonautonomous self-affine sets for $q>1$, which coincides with the upper bound of Theorem~\ref{Dq_ub}.

\begin{proof}[Proof of Theorem~\ref{thm_Dq1}]
Suppose that $d_q^-<d$. By Theorem \ref{Dq_ub}, it is sufficient to prove that for $q>1$
$$
\underline{D}_q(\mu^\w) \geq d_q^-.
$$

For non-integral $s_1<s_2<d_q^-$, by Proposition~\ref{pro_pdq_eq}, we have that for $ q>1$,
$$
\limsup_{k\to \infty} \sum_{\bu\in\Sigma^k} \psi^{s_2}(T_\bu)^{1-q}\mu(\mathcal{C}_\bu)^q<\infty.
$$
Then there exist constants $C>0$ and $K>1$ such that
for all $k>K$, we have that
$$
\sum_{\bu\in\Sigma^k} \psi^{s_2}(T_\bu)^{1-q}\mu(\mathcal{C}_\bu)^q \leq C.
$$
Since $s_1<s_2$, by \eqref{phis1}, we have that
$$
\psi^{s_1}(T_\bu)\geq \alpha_+^{-k(s_2-s_1)}\psi^{s_2}(T_\bu),
$$
and it immediately follows that
$$
\sum_{|\bu|=k}\psi^{s_1}(T_\bu)^{1-q}\mu(\mathcal{C}_\bu)^q \leq \alpha_+^{k(s_1-s_2)(1-q)} \sum_{\bu\in\Sigma^k} \psi^{s_2}(T_\bu)^{1-q}\mu(\mathcal{C}_\bu)^q.
$$
Let $\lambda=\alpha_+^{(s_1-s_2)(1-q)}$. Since $q>1$, it is clear that $0<\lambda<1$, and we have that
\begin{equation} \label{ineq_psilmd}
\sum_{|\bu|=k}\psi^{s_1}(T_\bu)^{1-q}\mu(\mathcal{C}_\bu)^q \leq C \lambda^k.
\end{equation}

For each non-integral $0<s<s_1$, by Proposition~\ref{calE}, it follows
that,
\begin{eqnarray}
&& \hspace{-0.8cm} \mathbb{E} \Big( r^{s(1-q)}\int \mu^\w(B(x,r))^{q-1}d\mu^\w(x)\Big)  = r^{(s_1-s)(q-1)} \mathbb{E} \Big( r^{s_1(1-q)}\int \mu^\w(B(x,r))^{q-1}d\mu^\w(x)\Big) \nonumber\\
&& \hspace{1cm}  \leq C r^{(s_1-s)(q-1)} \int\left[\int \ldots \int \psi^{s_1}(\bu_1,\ldots,\bu_n,\bv)^{-1} d\mu(\bu_1)\ldots d\mu(\bu_n)\right]^{(q-1)/n}d\mu(\bv), \label{ineq_ers}
\end{eqnarray}
for all sufficiently small $r$.
Write
$$
I=\int\left[\int \ldots \int \psi^{s_1}(\bu_1,\ldots,\bu_n,\bv)^{-1} d\mu(\bu_1)\ldots d\mu(\bu_n)\right]^{(q-1)/n}d\mu(\bv).
$$
For each $\bv$ we decompose the integral inside the square brackets as a sum of integrals taken over all $0 \leq k_1 < \ldots <k_p$ and all $m_1,\ldots,m_p\geq 1$ such that $m_1+\ldots+m_p=n$, all join classes $\mathcal{J}_1(\bv_1),\ldots,\mathcal{J}_p(\bv_p)$ where $\mathcal{J}_r(\bv_r)$ has root $\bv_r=\bv|_{k_r}$ and spread $m_r$. By Proposition~\ref{pro_intest},  we have
\begin{eqnarray*}
I &=& \int \bigg[ \sum_{\substack{k_1 < \ldots < k_p,\\ m_1 + \ldots + m_p = n,\\ \mathcal{J}_1, \ldots, \mathcal{J}_p}} \int_{\bigwedge(\bu_1,\ldots,\bu_{m_1}) \in\mathcal{J}_1(\bv_1)} \ldots  \\
&&\hspace{1cm} \int_{\bigwedge(\bu_{n-m_p+1},\ldots,\bu_n) \in\mathcal{J}_p(\bv_p)} \psi^s(\bu_1,\ldots,\bu_n,\bv)^{-1} d\mu(\bu_1)\ldots d\mu(\bu_n) \bigg]^{(q-1)/n} d\mu(\bv) \\
&\leq& \sum_{\substack{k_1 < \ldots < k_p,\\ m_1 + \ldots + m_p = n,\\ \mathcal{J}_1, \ldots, \mathcal{J}_p}} \int \bigg[ \int_{\bigwedge(\bu_1,\ldots,\bu_{m_1}) \in\mathcal{J}_1(\bv_1)} \ldots \\
&&\hspace{1cm} \int_{\bigwedge(\bu_{n-m_p+1},\ldots,\bu_n) \in\mathcal{J}_p(\bv_p)} \psi^s(\bu_1,\ldots,\bu_n,\bv)^{-1} d\mu(\bu_1)\ldots d\mu(\bu_n) \bigg]^{(q-1)/n} d\mu(\bv) \\
&\leq& \sum_{\substack{k_1 < \ldots < k_p,\\ m_1 + \ldots + m_p = n,\\ \mathcal{J}_1, \ldots, \mathcal{J}_p}} \prod_{l\in L}^{[n]}\Big(\sum_{|\bw|=l} \psi^s(T_\bw)^{1-q}\mu(\mathcal{C}_\bw)^q \Big)^{1/n},
\end{eqnarray*}
where the product is over the set of levels $L = {L(\mathcal{J}_r), \ldots, L(\mathcal{J}_p), k_r, \ldots, k_p}$ counted with repetitions. Let
\begin{eqnarray}
\hspace{1cm} M(l_1,\ldots,l_n) &=& \sharp\big\{(\mathcal{J}_1,\ldots,\mathcal{J}_p,  k_1,\ldots,k_p):1\leq p \leq n, 0\leq k_1<\ldots<k_p,\nonumber\\
&& \qquad \mathcal{J}_r \text{ is a join class with root at level } k_r\ \text{such that} \nonumber\\
&& \qquad  L(\mathcal{J}_1,\ldots,\mathcal{J}_p,  k_1,\ldots,k_p)=\{l_1,\ldots,l_n\}\big\}.
\end{eqnarray}
By~\eqref{ineq_psilmd}, we have
$$
I \leq \sum_{\substack{k_1 < \ldots < k_p,\\ m_1 + \ldots + m_p = n,\\ \mathcal{J}_1, \ldots, \mathcal{J}_p}} \prod_{l\in L}^{[n]}(C\lambda^l)^{1/n}
\leq C \sum_{0 \leq l_1 \leq \ldots \leq l_n} M(l_1,\ldots,l_n) \lambda^{(l_1+\ldots+l_n)/n},
$$

We write $M_0(l_1, \ldots, l_n)$ for the total number of join classes with root $\emptyset$ (the root of the tree $\Sigma^*$) and levels $l_1 \leq \ldots \leq l_n$. Note that $M_0(l_1)=1$. Every join set with levels $0 \leq l_1 \leq \ldots \leq l_n \leq l_{n+1}$ may be obtained
by joining a vertex at level $l_{n+1}$ to some vertex of a join set with levels $0 \leq l_1 \leq \ldots \leq l_n$ through a path in the tree $\Sigma^*$, and this may be done in at most $n$ inequivalent ways to within tree automorphism. Thus $M_0(l_1, \ldots, l_{n+1}) \leq n M_0(l_1,\ldots,l_n) \leq n!$. Since
$$
M(l_1,\ldots,l_n) \leq M_0(l_1,\ldots,l_n)\leq (n-1)!,
$$
it follows that
\begin{eqnarray*}
I  &\leq& C (n-1)! \sum_{l_1 \leq \ldots \leq l_n} \lambda^{(l_1+\ldots+l_n)/n}\leq C (n-1)! \sum_{k=1}^\infty P(k)\lambda^{k/n} < \infty, %\label{ineq_sumN}
\end{eqnarray*}
where $P(k)$ is the number of distinct ways of partitioning the integer $k$ into a sum of $n$ integers $k = l_1 + \ldots + l_n$ where $l_1 \leq \ldots \leq l_n$, and the series $\sum_{k=1}^\infty P(k)\lambda^{k/n}$ converges for $0 < \lambda < 1$ since $P(k)$ is polynomially bounded.
Combining this with ~\eqref{ineq_ers},  for all $0<s<s_1$, we have that
$$
\mathbb{E} \Big(\int r^{s(1-q)}\mu^\w(B(x,r))^{q-1}d\mu^\w(x)  \Big)\leq Cr^{(s_1-s)(q-1)}
$$
for all sufficiently small $r$, for some constant $C>0$.

By Markov's Inequality,
$$
\mathbf{P} \left\{ \int r^{s(1-q)}\mu^\w(B(x,r))^{q-1}d\mu^\w(x)\geq \sqrt{Cr^{(s_1-s)(q-1)}} \right\} \leq \sqrt{Cr^{(s_1-s)(q-1)}}.
$$
For all $0<\rho<1$, the Borel-Cantelli lemma implies that the sequence
$$
\int (\rho^k)^{s(1-q)}\mu^\w(B(x,\rho^k))^{q-1}d\mu^\w(x)
$$
converges to $0$ almost surely. Write
$$
f(r)= \int r^{s(1-q)}\mu^\w(B(x,r))^{q-1}d\mu^\w(x),
$$
and it is clear the $\lim_{k\to\infty} f(\rho^{k})=0$ almost surely.

Given $\rho>0$.  For every sufficiently small $r>0$, the exists a unique integer $k>0$ such that $\rho^{k+1}\leq r<\rho^k$. For $q>1$,  it follows that  
$$
\rho^{s(q-1)} f(\rho^{k+1}) \leq f(r) \leq \frac{f(\rho^k)}{\rho^{s(q-1)}},
$$
and this implies that
$$
\lim_{r\to 0}\int r^{s(1-q)}\mu^\w(B(x,r))^{q-1} d\mu^\w(x)=\lim_{r\to 0} f(r)=0
$$
almost surely. By Proposition \ref{pro_eqdefDq}, this implies that   $\underline{D}_q\geq s$  for all $s<d_q^-$, and we have that
$$
\underline{D}_q(\mu^\w)\geq\min\{d_q^-,d\}
$$
almost surely. Combining this with Theorem~\ref{Dq_ub}, the conclusion holds.

\end{proof}

Finally, we explore the generlized $q$-dimension of measures supported on almost self-affine sets. Recall that  almost self-affine set $E^\omega$ is  a  nonautonomous affine set with random translations  where all  $\Xi_{k}$ are identical, that is,
$$
\Xi_{k}=\Xi=\{T_{1},T_{2},\ldots,T_{n_0}\}.
$$
Falconer first studied generalized $q$-dimension on these sets, and he provided the generalized $q$-dimension formula for $q>1$, see ~\cite{Falco10} for details. We improve the conclusion to $q\geq 1$ by applying Theorem~\ref{thm_Dq1} and Proposition \ref{Dq_eq}.
\begin{proof}[Proof of Corollary~\ref{Dq_aa}]
For $q>1$,  the fact $d_q^-=d_q^+=d_q$ was proved by Falconer in ~\cite[Lemma 8.1]{Falco10}, and by Theorem~\ref{thm_Dq1} and Proposition \ref{Dq_eq}, the generalized $q$-dimension $D_q(\mu^\w)$ exists and
$$
D_q(\mu^\w)=\min\{d_q,d\}.
$$

Since $\Xi_{k}=\Xi=\{T_{1},T_{2},\ldots,T_{n_0}\}$, the singular value function $\psi^s$ is submultiplicative, and by the subadditive ergodic theorem (see \cite{Ste}), the following limit exists  for almost all $\bu \in \Sigma^\infty$, i.e.,
\begin{eqnarray*}
\lim_{k\to \infty} \frac{1}{k}\log\left(\psi^s(T_{\bu|k})^{-1} p_{\bu|k} \right)&=&\lim_{k\to \infty} \frac{1}{k}\int \log\left(\psi^s(T_{\bu|k})^{-1} p_{\bu|k} \right) d\mu(\bu)\\
&=& \lim_{k\to \infty} \frac{1}{k} \sum_{\bu\in \Sigma^k} p_{\bu} \log\left(\psi^s(T_{\bu})^{-1} p_{\bu} \right).
\end{eqnarray*}
We write
$$
f(s)=\lim_{k\to \infty} \frac{1}{k}\log\left(\psi^s(T_{\bu|k})^{-1} p_{\bu|k} \right).
$$
By the similar argument as before, it is clear that $f(s)$ is strictly monotonic increasing and continuous in $s$. Therefore $d_1$ is the unique solution to $f(s)=0$.

It remains to show that  $d_1^-=d_1^+=d_1$, and  the generalized $q$-dimension $D_1(\mu^\w)$ exists and
$$
D_1(\mu^\w)=\min\{d_1,d\}.
$$
Since by Theorem~\ref{Dq_ub}, we have that
$$
\underline{D}_1(\mu^\w) \leq d_1^-, \qquad \textit{ and }\qquad  \underline{D}_1(\mu^\w) \leq \overline{D}_1(\mu^\w) \leq d_1^+,
$$
and the conclusion follows if the inequalities  $d_1\leq \underline{D}_1(\mu^\w) $ and $d_1^+ \leq d_1$ hold.

We first prove that  $d_1\leq \underline{D}_1(\mu^\w) $. Arbitrarily choosing $s<d_1$, since  $f(s)$ is strictly monotonic increasing and continuous in $s$, we have that
$$
f(s)=\lim_{k\to \infty} \frac{1}{k}\log\left(\psi^s(T_{\bu|k})^{-1} p_{\bu|k} \right) < 0,
$$
for almost all $\bu\in \Sigma^\infty$, and it follows that
$$
\sum_{k=0}^{\infty} \psi^s(T_{\bu|k})^{-1} p_{\bu|k} < \infty.
$$

For all non-integral $s$ such that $0 < s < d_1$, by Lemma~\ref{lemE}, Fubini's theorem and tower property,
\begin{eqnarray*}
\mathbb{E} \big( \mu^{\w}(B(\Pi^{\w}(\bu),r)) \big)
&\leq& \mathbb{E}\Big( \int_{\Sigma_\infty}r^s |\Pi^{\w}(\bu)-\Pi^{\w}(\bv)|^{-s}\Big)
d\mu(\bv)  \\
&\leq& r^s
\int_{\Sigma_\infty}\frac{C}{\psi^s(T_{\bu\wedge\bv})} d\mu(\bv)  \\
&\leq& C r^s \sum_{k=0}^{\infty}\psi^s(T_{\bu|k})^{-1}
\mu(\mathcal{C}_{\bu|k}) \\
&\leq& C_1 r^s,
\end{eqnarray*}
for almost all $\bu$, where $C_1$ is a constant that depends on $\bu$. Hence for $t<s<d_1$,  
$$
\mathbb{E} \big( r^{-t} \mu^{\w}(B(\Pi^{\w}(\bu),r)) \big) \leq \sqrt{C_1 r^{s-t}},
$$
for all $0<r<1$, for almost all $\bu$. By Markov's Inequality,
$$
\mathbf{P} \big\{ r^{-t} \mu^{\w}(B(\Pi^{\w}(\bu),r)) \geq \sqrt{C_1 r^{s-t}} \big\} \leq C_1 r^{s-t}.
$$
For all $0<\rho<1$, the Borel-Cantelli lemma implies that the sequence
$$
\rho^{-kt} \mu^{\w}(B(\Pi^{\w}(\bu),\rho^k))
$$   
converges to $0$ almost surely. Write 
$$
g(r)=r^{-t} \mu^{\w}(B(\Pi^{\w}(\bu),r))
$$
and it is clear the $\lim_{k\to\infty} g(\rho^{k})=0$ almost surely. 

Given $\rho>0$.  For every sufficiently small $r>0$, the exists a unique integer $k>0$ such that $\rho^{k+1}\leq r<\rho^k$. For $q>1$,  we have  that
$$
\rho^t g(\rho^{k+1}) \leq g(r) \leq \frac{g(\rho^k)}{\rho^t},
$$
and this implies that
$$
\lim_{r\to 0} r^{-t} \mu^{\w}(B(\Pi^{\w}(\bu),r)) = \lim_{r\to 0} g(r)=0
$$
almost surely. Hence for almost all $\bu\in \Sigma^\infty$,
$$
\liminf_{r \to 0} \frac{\log \mu^{\w}(B(\Pi^{\w}(\bu),r))}{\log r} \geq t
$$
almost surely. It follows that for almost all $\w$, for any sequence $\{b_n\}$ such that $\lim_{n\to \infty} b_n = 0$,
$$
\liminf_{n \to \infty} \frac{\log \mu^{\w}(B(\Pi^{\w}(\bu),b_n))}{\log b_n} \geq t
$$
for almost all $\bu\in \Sigma^\infty$. Then by Fatou's Lemma, for almost all $\w$, 
$$
\liminf_{n \to \infty} \int_{\Sigma^\infty} \frac{\log \mu^{\w}(B(\Pi^{\w}(\bu),b_n))}{\log b_n} d\mu(\bu) \geq \int_{\Sigma^\infty} \liminf_{n \to \infty} \frac{\log \mu^{\w}(B(\Pi^{\w}(\bu),b_n))}{\log b_n} d\mu(\bu) \geq t.
$$
Since the sequence $\{b_n\}$ is arbitrarily chosen, we have that for almost all $\w$, 
$$
\underline{D}_1(\mu^\w) = \liminf_{r \to 0} \int_{\Sigma^\infty} \frac{\log \mu^{\w}(B(\Pi^{\w}(\bu),r))}{\log r} d\mu(\bu) \geq t.
$$          
Since $t<s<d_1$ is arbitrarily chosen, we obtain that
\begin{equation}\label{ineq_D1geqd1}
\underline{D}_1(\mu^\w) \geq d_1
\end{equation}
almost surely.

\iffalse
For all non-integral $s$ such that $0 < s < d_1$, by Jensen's inequality and Fubini's theorem, we have that
\begin{eqnarray*}
\mathbb{E} \bigg( \exp\bigg[ \int \log \mu^\w(B(x,r)) d\mu^\w(x) \bigg] \bigg)  &\leq&  \mathbb{E} \bigg( \int_{\Sigma^\infty}  \mu^\w(B(\Pi^\w(\bu),r)) d\mu(\bu) \bigg) \\
&\leq& \int_{\Sigma^\infty} \mathbb{E}\big(\mu^\w(B(\Pi^\w(\bu),r))\big) d\mu(\bu) \\
&=& \int_{\Sigma^\infty} \mathbb{E}\big( \mu\{\bv: |\Pi^\w(\bu)-\Pi^\w(\bv)| \leq r\} \big) d\mu(\bu) \\
&\leq& \int_{\Sigma^\infty} \mathbb{E}\bigg( \int_{\Sigma^\infty} \frac{r^s d\bv}{|\Pi^\w(\bu)-\Pi^\w(\bv)|^s}\bigg) d\mu(\bu) \\
&=& r^s \int_{\Sigma^\infty} \int_{\Sigma^\infty} \mathbb{E}(|\Pi^\w(\bu)-\Pi^\w(\bv)|^{-s}) d\mu(\bv) d\mu(\bu).
\end{eqnarray*}
Then by Lemma~\ref{lemE} and  tower property, we have
\begin{eqnarray*}
\mathbb{E} \bigg(r^{-s} \exp\bigg[ \int \log \mu^\w(B(x,r)) d\mu^\w(x) \bigg] \bigg) &\leq& C \int_{\Sigma^\infty} \int_{\Sigma^\infty} \psi^s(T_{\bu\wedge\bv})^{-1} d\mu(\bv) d\mu(\bu)  \\
&\leq& C \int_{\Sigma^\infty} \sum_{k=0}^{\infty} \psi^s(T_{\bu|k})^{-1} p_{\bu|k} d\mu(\bu) \\
&\leq& C_1.
\end{eqnarray*}
Thus for almost all $\w$,
$$\limsup r^{-s} \exp\Big[ \int \log \mu^\w(B(x,r)) d\mu^\w(x) \Big]<\infty,$$
and by Proposition~\ref{pro_eqdefDq}, it immediately follows that  $\underline{D}_1(\mu^\w) \geq s$ for all $s<d_1$. Since $s$ is arbitrarily chosen, we obtain that
\begin{equation}\label{ineq_D1geqd1}
\underline{D}_1(\mu^\w) \geq d_1.
\end{equation}
\fi

Next, we prove that $d_1^+ \leq d_1$.
Given $s>d_1$. Since  $f(s)$ is strictly monotonic increasing and continuous in $s$, we have that
$$
f(s)=\lim_{k\to \infty} \frac{1}{k} \sum_{\bu\in \Sigma^k} p_{\bu} \log\left(\psi^s(T_{\bu})^{-1} p_{\bu} \right)>0,
$$
and $f(s)$ only depends on $s$. Fix $0<\gamma<f(s)$, there exists $K>0$ such that for all $k\geq K$,
\begin{equation} \label{fs>gma}
\sum_{\bu\in \Sigma^k} p_{\bu} \log\left(\psi^s(T_{\bu})^{-1} p_{\bu} \right) >K \gamma.
\end{equation}

For all $\bu\in \Sigma^*$, we have
\begin{equation}\label{ineq_sumpuj}
\sum_{j=1}^{n_0} p_{\bu j} \log \big(\psi^s(T_{\bu j})^{-1} p_{\bu j}\big) \geq p_{\bu} \log \left(\psi^s(T_{\bu})^{-1} p_{\bu}\right) + p_{\bu} \Big(\sum_{j=1}^{n_0} p_j \log\left(\psi^s(T_j)^{-1} p_j\right)\Big)
\end{equation}

For every $r>0$,  recall that
$$
\Sigma^*(s, r)=\{\mathbf{u}=u_1\ldots u_k \in \Sigma^*: \alpha_m(T_\mathbf{u})  \leq r < \alpha_m(T_{\bu^-}) \}.
$$
Let
$$
K_1=\max\{|\bu|: \bu\in \Sigma^*(s,r)\},\qquad K_2=\min\{|\bu|: \bu\in \Sigma^*(s,r)\}.
$$
It is clear that  $K_2 $ tends to $\infty$ as $r$ tends to $0$. Hence, we assume that  $K_2 \geq K$.

For each $\bu\in \Sigma^*(s,r)$ with length $K_1$, it is clear that $\bu^-j\in \Sigma^*(s,r)$ for all $j\in\{1,\ldots,n_0\}$. If $\sum_{j=1}^{n_0} p_j\log \left(\psi^s(T_j)^{-1} p_j\right) \geq 0$, then by ~\eqref{ineq_sumpuj}, we have
\begin{equation*}
\sum_{j=1}^{n_0} p_{\bu^- j} \log \Big(\psi^s(T_{\bu^- j})^{-1} p_{\bu^- j}\Big)
\geq p_{\bu^-} \log \left(\psi^s(T_{\bu^-})^{-1} p_{\bu^-}\right).
\end{equation*}
By repeating this process $(K_1-K_2)$ times, we have
$$
\sum_{\bu\in \Sigma^*(s,r)}p_{\bu} \log \left(\psi^s(T_{\bu})^{-1} p_{\bu}\right) \geq \sum_{\bu\in \Sigma^{K_2}} p_{\bu} \log \left(\psi^s(T_{\bu})^{-1} p_{\bu}\right)>K_2 \gamma.
$$

If $\sum_{j=1}^{n_0} p_j \log \left(\psi^s(T_j)^{-1} p_j\right) < 0$, then we write $W_K(s,r)$ for the collection of all  finite words $\bv$ such that:
\begin{itemize}
\item[(i).] $|\bv|$ is a multiple of $K$;
\item[(ii).] $\alpha_m(T_\bv) \geq r$;
\item[(iii).] $\alpha_m(T_{\bv \bv'}) < r$ for some $\bv' \in \Sigma^K$.
\end{itemize}
Note that $W_K(s,r)$ gives a cover of $\Sigma^\infty$, but $W_K(s,r)$ is not a cut set.
For $\bu, \bv\in\Sigma^*$ such that $|\bu|>|\bv|$, we say $\bu$ is an extension  of $\bv$  if there exists an integer $k>0$ such that $\bv=\bu|k$. We write  $\bu \succeq \bv$ if either $\bv=\bu$ or $\bu$ is an extension  of $\bv$.  Let $W'_K(s,r)$ be a subset of  $W_K(s,r)$ consisting  of words which are not extensions of  other word in $W_K(s,r)$, i.e,
$$
W'_K(s,r)=\{\bv\in W_K(s,r) : \bv|k \notin  W_K(s,r) \textit { for all } k>0\}.
$$
 It is clear that $W'_K(s,r)$ is a cut set. % that is, $\Sigma^\infty = \bigcup_{\bv \in W'_K(s,r)} \mathcal{C}_\bv$.

For every  $\bu \in \Sigma^*(s,r)$, there exists a unique $\bv \in W'_K(s,r)$ such that $\bu \succeq \bv$. Write
$$
\Sigma^*(s,r,\bv)=\{\bu \in \Sigma^*(s,r): \bu \succeq \bv\},
$$
and it is clear that
$$
\mathcal{C}_\bv=\bigcup_{\bu \in \Sigma^*(s,r,\bv)} \mathcal{C}_\bu, \qquad \textit{and} \qquad  \Sigma^*(s,r)= \bigcup_{\bv \in W'_K(s,r)}  \Sigma^*(s,r,\bv).
$$

Given $\bv \in W'_K(s,r)$, by (iii), we have
$$
\alpha_-^K \alpha_m(T_\bv) \leq \alpha_m(T_{\bv \bv'}) < r
$$
for some $\bv' \in \Sigma^K$, and it is clear that $\alpha_m(T_\bv) < \alpha_-^{-K} r$. Hence  for all  $\bu \in \Sigma^*(s,r,\bv)$, we have that
$$
\alpha_+^{|\bu|-1-|\bv|} \alpha_-^{-K} r > \alpha_+^{|\bu|-1-|\bv|} \alpha_m(T_\bv) \geq \alpha_m(T_{\bu^-}) > r.
$$
This implies that  $|\bu| - |\bv| < C(K)$ where $C(K)=\frac{K\log \alpha_-}{\log \alpha_+}+1$ is a constant and only depends on $K$. Similarly, we write  $K_1(\bv)=\max\{|\bu|: \bu \in \Sigma^*(s,r,\bv)\}$, and it is clear that
\begin{equation}\label{ineq_K1}
K_1(\bv) < |\bv| + C(K).
\end{equation}

For each $\bu \in \Sigma^*(s,r,\bv)$ with length $K_1(\bv)$, it is clear that $\bu^-j\in \Sigma^*(s,r)$ for all $j\in\{1,\ldots,n_0\}$, and it implies that
$$
\sum_{j=1}^{n_0} p_{\bu^- j} \log \Big(\psi^s(T_{\bu^- j})^{-1} p_{\bu^- j}\Big) \geq p_{\bu^-} \log \left(\psi^s(T_{\bu^-})^{-1} p_{\bu^-}\right) + p_{\bu^-} \Big(\sum_{j=1}^{n_0} p_j \log\left(\psi^s(T_j)^{-1} p_j\right)\Big).
$$
Since $\sum_j p_j \log\left(\psi^s(T_j)^{-1} p_j\right) < 0$ and $p_{\bu^-}\leq p_{\bv} $  for all $\bu \in \Sigma^*(s,r,\bv)$, we have that
\begin{eqnarray*}
&& \hspace{-2cm} \sum_{\bu \in \Sigma^*(s,r,\bv) \atop |\bu|=K_1(\bv)} p_{\bu} \log \Big(\psi^s(T_{\bu})^{-1} p_{\bu}\Big) = \sum_{|\bu^-|=K_1(\bv)-1} \sum_{j=1}^{n_0} p_{\bu^- j} \log \Big(\psi^s(T_{\bu^- j})^{-1} p_{\bu^- j}\Big) \\
&\geq& \sum_{|\bu^-|=K_1(\bv)-1 \atop \bu \in \Sigma^*(s,r,\bv)} \left(p_{\bu^-} \log \left(\psi^s(T_{\bu^-})^{-1} p_{\bu^-}\right) + p_{\bu^-} \Big(\sum_{j=1}^{n_0} p_j \log\left(\psi^s(T_j)^{-1} p_j\right)\Big) \right) \\
&\geq& \sum_{|\bu^-|=K_1(\bv)-1 \atop \bu \in \Sigma^*(s,r,\bv)} \left(p_{\bu^-} \log \left(\psi^s(T_{\bu^-})^{-1} p_{\bu^-}\right)\right) + p_{\bv} \Big(\sum_{j=1}^{n_0} p_j \log\left(\psi^s(T_j)^{-1} p_j\right)\Big).
\end{eqnarray*}
By repeating this process $K_1(\bv)-|\bv|$ times and ~\eqref{ineq_K1}, we have
\begin{eqnarray}
&& \hspace{-2cm} \sum_{\bu \in \Sigma^*(s,r,\bv)} p_{\bu} \log \Big(\psi^s(T_{\bu})^{-1} p_{\bu}\Big) \nonumber \\
&\geq& p_{\bv} \log \left(\psi^s(T_{\bv})^{-1} p_{\bv}\right) + (K_1(\bv) - |\bv|) p_{\bv} \Big(\sum_{j=1}^{n_0} p_j \log\left(\psi^s(T_j)^{-1} p_j\right)\Big) \nonumber \\
&>& p_{\bv} \log \left(\psi^s(T_{\bv})^{-1} p_{\bv}\right) + C(K) p_{\bv} \Big(\sum_{j=1}^{n_0} p_j \log\left(\psi^s(T_j)^{-1} p_j\right)\Big). \label{ineq_sumpv}
\end{eqnarray}
Summing up ~\eqref{ineq_sumpv}, we have
\begin{eqnarray}
&& \hspace{-2cm} \sum_{\bu \in \Sigma^*(s,r)} p_{\bu} \log \Big(\psi^s(T_{\bu})^{-1} p_{\bu}\Big)= \sum_{\bv \in W'_K(s,r)} \sum_{\bu \in \Sigma^*(s,r,\bv)} p_{\bu} \log \Big(\psi^s(T_{\bu})^{-1} p_{\bu}\Big) \nonumber \\
&\geq&  \sum_{\bv \in W'_K(s,r)} p_{\bv} \log \left(\psi^s(T_{\bv})^{-1} p_{\bv}\right) + C(K) \Big(\sum_{j=1}^{n_0} p_j \log\left(\psi^s(T_j)^{-1} p_j\right)\Big). \label{ineq_sumpu}
\end{eqnarray}
We write
$$
K_3=\max\{|\bv|: \bv\in W'_K(s,r)\},\qquad K_4=\min\{|\bv|: \bv\in W'_K(s,r)\},
$$
and  $K_3$ and $K_4$ are both multiples of $K$. It is clear that $K_4 \geq K$ and $K_4$ tends to $ \infty$ as $r$ goes to $ 0$. Since $W'_K(s,r) \subset W_K(s,r)$ is a cut set, by (i), for every $\bw\in W'_K(s,r)$ with $|\bw|=K_3$, we write $\bw=\bv\bv'$, where $|\bv'|=K$. It is clear that $\bv \bv'' \in W'_K(s,r)$ for all $\bv''\in \Sigma^K$. By ~\eqref{fs>gma} and ~\eqref{ineq_sumpuj},  we have
\begin{eqnarray*}
&& \hspace{-1cm} \sum_{\bv''\in \Sigma^K} p_{\bv \bv''} \log \Big(\psi^s(T_{\bv \bv''})^{-1} p_{\bv \bv''}\Big) \geq p_{\bv} \log \Big(\psi^s(T_{\bv})^{-1} p_{\bv}\Big) + \sum_{\bv''\in \Sigma^K} p_{\bv''} \log\left(\psi^s(T_{\bv''})^{-1} p_{\bv''} \right) \\
&&\hspace{4.9cm}> p_{\bv} \log \Big(\psi^s(T_{\bv})^{-1} p_{\bv}\Big)+K\gamma.
\end{eqnarray*}
Since $K_4$ is a multiple of $K$, by applying above argument repeatedly, we have
$$
\sum_{\bv \in W'_K(s,r)} p_{\bv} \log \left(\psi^s(T_{\bv})^{-1} p_{\bv}\right) > \sum_{\bv \in \Sigma^{K_4}} p_{\bv} \log \left(\psi^s(T_{\bv})^{-1} p_{\bv}\right) > K_4 \gamma.
$$
Combining this with ~\eqref{ineq_sumpu}, we obtain that
$$
\sum_{\bu \in \Sigma^*(s,r)} p_{\bu} \log \Big(\psi^s(T_{\bu})^{-1} p_{\bu}\Big) > K_4 \gamma + C(K) \Big(\sum_{j=1}^{n_0} p_j \log\left(\psi^s(T_j)^{-1} p_j\right)\Big).
$$
Since $K_4$ tends to $\infty$ as $r$ goes to $0$, we have
$$
\liminf_{r\to 0} \sum_{\bu \in \Sigma^*(s,r)} p_{\bu} \log \Big(\psi^s(T_{\bu})^{-1} p_{\bu}\Big) = \infty.
$$
It follows that $s>d_1^+$ for all  $s>d_1$, and  we have
\begin{equation}\label{ineq_d1+leqd1}
d_1^+ \leq d_1.
\end{equation}

Combining \eqref{ineq_D1geqd1} and \eqref{ineq_d1+leqd1} together, we obtain that  $d_1=d_1^+=d_1^-$, and $D_1(\mu^\w)=d_1$, and the conclusion follows.

\end{proof}

\section{ Generalized $q$-dimensions of measures on nonautonomous affine sets with finitely many translations}\label{sec_6}

The following lemma was proved by Falconer in~\cite[Lemma 2.1]{Falco88} which is very useful in studying fractals generated by affine mappings. Let $\overline{B(0,\rho)}$ be the closed ball in $\R^d$ with centre at the origin and radius $\rho$.
\begin{lem}\label{lemtphi}
Let $s$ be non-integral with $0 < s< d$.
Then there exists constant $C<\infty$ dependent on $d, s$ and $\rho$ such that for all non-singular
linear transformations $T\in {\mathcal{ L}}(\R^d,\R^d)$,
$$
\int_{\overline{B(0,\rho)}} \frac{dx}{|Tx|^{s}}
\leq \frac{C}{\psi^{s}(T)}.
$$
\end{lem}

Finally, we study the nonautonomous affine fractals with finitely many translations. Let $\Gamma=\{a_1,\ldots, a_\tau\}$ be a finite collection of translations, where $a_1,\ldots, a_\tau$ are regarded later as variables in $\R^d$. Suppose that the translation of $\Psi_{u_j}$ is an element of $\Gamma$, that is,
 $$
\Psi_{u_j}(x)=T_{j,u_j}x+ \w_{u_1\ldots u_j}, \qquad \w_{u_1\ldots u_j}\in \Gamma,
  $$
 for $j=1,2,\ldots, k$.
We write  $E^\ba$ for the nonautonomous affine set  to emphasize the dependence on these special translations in $\Gamma$. The following proof is inspired by ~\cite{Falco88,Sol,GM}.

\begin{proof}[Proof of Theorem~\ref{dimmua}]
For $\bu, \bv\in \Sigma^\infty$, we assume that $|\bu \wedge \bv|=n$. Without loss of generality, suppose that $\w_{\bu|{n+1}}=a_1$, $w_{\bv|{n+1}}=a_2$. Then
\begin{eqnarray*}
\Pi^\ba(\bu)-\Pi^\ba(\bv)&=& T_{\bu\wedge\bv}\Big(a_1-a_2+ (T_{n+1,u_{n+1}}\w_{\bu|n+2}+T_{n+1,u_{n+1}}T_{n+2,u_{n+2}} \w_{\bu|{n+3}}+\ldots) \\
&& \qquad \qquad \qquad-(T_{n+1,v_{n+1}}\w_{\bv|n+2}  +T_{n+1,v_{n+1}}T_{n+2,v_{n+2}} \w_{\bu|{n+3}} +\ldots)\Big) \\
&=& T_{\bu\wedge\bv}(a_1-a_2+ H(\ba)),
\end{eqnarray*}
where $H$ is a linear map from $\R^{\tau d}$ to $\R^d$. We may write $H(\ba)=\sum_{a_j\in\Gamma \atop 1\leq j\leq \tau} H_j(a_j)$.

We write
\begin{eqnarray*}
\eta &=&\sup\{\|T_{k,j}\| : T_{k,j}\in \Xi_k, 0<j\leq n_k,  k>0 \},
\end{eqnarray*}
and we have that  $\eta < \frac{1}{2}.$  Let
$$
m=\inf\{k\geq 2:\w_{\bu|{n+k}}=\w_{\bv|{n+k}}\}.
$$
If $m=\infty$, it is straightforward that
$$
\|H_1\|\leq \sum_{k=2}^\infty \eta^{k-1} < 1.
$$
Otherwise for $m<\infty$, it is clear that $\w_{\bu|{n+m}}$ and $\w_{\bv|{n+m}}$ do not  equal to $a_1$ and $a_2$ simultaneously. We assume that $\w_{\bu|{n+m}}=\w_{\bv|{n+m}}\neq a_1$. Since $\eta <\frac{1}{2}$,  it follows that
\begin{eqnarray*}
\|H_1\|&\leq& \sum_{k=2}^{m-1} \eta^{k-1}+ \sum_{k=m+1}^{\infty} 2\eta^{k-1}<1. %\\
%&<& \sum_{k=2}^{m-1} \eta^{k-1}+ \sum_{k=m+1}^{\infty} \eta^{k-2}\\
%&=& \sum_{k=2}^{m-1} \eta^{k-1}+ \sum_{k=m}^{\infty} \eta^{k-1} \\
%& < &1.
\end{eqnarray*}
This implies that  the linear transformation $I+H_1$ is invertible. We define  the  invertible linear transformation by
\begin{equation}\label{inlitr}
y=a_1-a_2+ H(\ba), \ a_2=a_2,\ldots,\ a_\tau=a_\tau.
\end{equation}

For  non-integral  $s\leq d$,  by Lemma~\ref{lemtphi} and \eqref{inlitr}, it follows that
\begin{eqnarray} \label{ineq_ssiphi}
\int_{\overline{\mathbf{B}(0,\rho)}}\frac{d\ba}{|\Pi^\ba(\bu)-\Pi^\ba(\bv)|^s} &=& \int_{\overline{\mathbf{B}(0,\rho)}}\frac{d\ba}{|T_{\bu\wedge\bv} (a_1-a_2+ H(\ba))|^s}         \nonumber \\
&\leq&C_1 \int \ldots \int_{y\in B_{(2+k)\rho} \atop a_j\in B_\rho} \frac{dy da_2 \ldots}{|T_{\bu\wedge\bv}(y)|^s}     \nonumber\\
&\leq& \frac{C}{\psi^{s}(T_{\bu \wedge \bv})},
\end{eqnarray}
where $C_1$ and $C$ are constants dependent on $d, s$ and $\rho$.

Since $1<q\leq 2$, we have that $0<q-1\leq 1$. By Jensen's inequality,
\begin{eqnarray*}
 \hspace{2cm} && \hspace{-4cm}\int_{\overline{\mathbf{B}(0,\rho)}} \int \mu^\ba(B(x,r))^{q-1} d\mu^\ba(x)d\ba \leq \int_{\overline{\mathbf{B}(0,\rho)}} \int_{\Sigma^\infty} \mu(B(\Pi^\ba(\bu),r))^{q-1} d\mu(\bu)d\ba  \nonumber \\
   &\leq& C_2 \int_{\Sigma^\infty} \Big[\int_{\overline{\mathbf{B}(0,\rho)}} \mu(B(\Pi^\ba(\bu),r))d\ba\Big]^{q-1} d\mu(\bu) \nonumber\\
   &=& C_2 \int_{\Sigma^\infty} \Big[\int_{\overline{\mathbf{B}(0,\rho)}} \mu\{\bv: |\Pi^\ba(\bu)-\Pi^\ba(\bv)| \leq r\} d\ba\Big]^{q-1} d\mu(\bu).
 \end{eqnarray*}
Using ~\eqref{ineq_ssiphi}, we obtain that
\begin{eqnarray}
    \hspace{2cm} && \hspace{-4cm}\int_{\overline{\mathbf{B}(0,\rho)}} \int \mu^\ba(B(x,r))^{q-1} d\mu^\ba(x)d\ba \leq \int_{\Sigma^\infty} \Big[\int_{\overline{\mathbf{B}(0,\rho)}} \int_{\Sigma^\infty} \frac{r^s d\bv d\ba}{|\Pi^\ba(\bu)-\Pi^\ba(\bv)|^s} \Big]^{q-1} d\mu(\bu)    \nonumber \\
   &\leq& C_3 r^{s(q-1)} \int_{\Sigma^\infty} \Big[ \int_{\Sigma^\infty} \psi^s(T_{\bu\wedge\bv})^{-1} d\bv \Big]^{q-1} d\mu(\bu) \nonumber\\
   &\leq& C_3 r^{s(q-1)} \int_{\Sigma^\infty} \Big[ \sum_{k=0}^{\infty} \psi^s(T_{\bu|k})^{-1} \mu(\mathcal{C}_{\bu|k}) \Big]^{q-1} d\mu(\bu)    \nonumber \\
   &\leq& C_3 r^{s(q-1)} \int_{\Sigma^\infty}  \sum_{k=0}^{\infty} \psi^s(T_{\bu|k})^{1-q} \mu(\mathcal{C}_{\bu|k})^{q-1} d\mu(\bu).    \label{ineq_ei}
\end{eqnarray}

Recall that
$$
d_q^-=\sup \Big\{s: \sum_{k=1}^\infty \sum_{\bu\in\Sigma^k} \psi^s(T_\bu)^{1-q}\mu(\mathcal{C}_\bu)^q<\infty\Big\}.
$$
Arbitrarily choosing a non-integral $0<s<d_q^-$, we have that
$$
\sum_{k=1}^\infty \sum_{\bu\in\Sigma^k} \psi^s(T_\bu)^{1-q}\mu(\mathcal{C}_\bu)^q<\infty.
$$
By \eqref{ineq_ei}, it follows that  for almost all $\ba\in \overline{\mathbf{B}(0,\rho)}$,
\begin{eqnarray*}
r^{s(1-q)} \int \mu^\ba(B(x,r))^{q-1} d\mu^\ba(x)
&\leq& C_3 \sum_{k=0}^{\infty} \int_{\Sigma^\infty} \psi^s(T_{\bu})^{1-q} \mu(\mathcal{C}_\bu)^{q-1} \mu(\bu)\\
&=& C_3 \sum_{k=0}^{\infty} \sum_{\bu\in \Sigma^k} \psi^s(T_{\bu})^{1-q} \mu(\mathcal{C}_\bu)^q \\
&<& \infty.
\end{eqnarray*}

By Proposition \ref{pro_eqdefDq}, we obtain that
$\underline{D}_q(\mu^\ba) \geq s$ for all $s<d_q^-$. Since  $s$ is arbitrarily chosen, we have that $\underline{D}_q(\mu^\ba) \geq d_q^-$.
\end{proof}

\begin{proof}[Proof of Corollary~\ref{cor_nasid}]
The proof is similar to  Corollary~\ref{Dq_aa}, and we omit it.
\end{proof}

\end{document}